\documentclass[english,leqno,10pt,a4paper]{article}
\usepackage{amsmath,amstext, amsthm, amssymb}

\usepackage{pifont}
\usepackage[hypertex]{hyperref}
\usepackage{srcltx}
\usepackage[totalheight=24 true cm, totalwidth=17 true cm]{geometry}
\usepackage[english]{babel}
\usepackage[all]{xy}
\usepackage{color}

\newtheorem{thm}{Theorem}[section]
\newtheorem{cor}[thm]{Corollary}
\newtheorem{lemma}[thm]{Lemma}
\newtheorem{prop}[thm]{Proposition}

\theoremstyle{remark}
\theoremstyle{definition}
\newtheorem{defn}[thm]{Definition}

\newtheorem{rmk}[thm]{Remark}
\newtheorem{exa}[thm]{Example}

\numberwithin{equation}{section}

\def\beq{\begin{equation}}
\def\eeq{\end{equation}}

\def\crash#1{}

\def\Z{{\mathbb Z}}
\def\Q{{\mathbb Q}}

\def\C{{\mathbb C}}

\def\P{{\mathbb P}}

\def\l{\left}
\def\r{\right}
\def\[[{\l[\l[}
\def\]]{\r]\r]}
\def\p{\prime}

\def\sgq{\sigma_q}

\def\Sgq{\Sigma_q}

\def\cf{\emph{cf. }}
\def\ie{\emph{i.e. }}
\def\ds{\displaystyle}

\def\cA{{\mathcal A}}

\def\cC{{\mathcal C}}
\def\cD{{\mathcal D}}
\def\cE{{\mathcal E}}
\def\cF{{\mathcal F}}
\def\cM{{\mathcal M}}
\def\cN{{\mathcal N}}
\def\cH{{\mathcal H}}
\def\cI{{\mathcal I}}
\def\cJ{{\mathcal J}}
\def\cL{{\mathcal{ L}}}
\def\cO{{\mathcal O}}

\def\cT{{\mathcal T}}

\def\cW{{\mathcal W}}

\def\wtilde{\widetilde}

\def\ul{\underline}
\def\ol{\overline}

\def\a{\alpha}
\def\be{\beta}
\def\de{\delta}
\def\ga{\gamma}

\def\Ga{\Gamma}
\def\la{\lambda}
\def\La{\Lambda}

\def\GL{{\rm GL}}

\def\cG{{\mathcal G}}

\def\Gal#1{{\cG al}^{alg} (#1)}
\def\Galt#1{\wtilde{{\cG al}^{alg} (#1)}}

\def\Kol#1{{\mathcal{K} ol}(#1)}
\def\Kolt#1{\wtilde{\mathop{\mathcal{K}ol}(#1)}}

\def\Zar#1{{\mathcal{Z} ar}(#1)}
\def\Zart#1{\wtilde{\mathop{\mathcal{Z}ar}(#1)}}

\def\Galan#1{{\mathcal{G}al}(#1)}

\def\J{{\mathbb J}}

\author{Lucia Di Vizio\thanks{Lucia DI VIZIO,
Laboratoire de Mathématiques UMR 8100, CNRS,
Universit\'e de Versailles-St Quentin,
45 avenue des États-Unis
78035 Versailles cedex, France.
{\tt divizio@math.cnrs.fr}}~
and
Charlotte Hardouin\thanks{Charlotte HARDOUIN, Institut de Math\'{e}matiques de Toulouse,
118 route de Narbonne,
31062 Toulouse Cedex 9, France.
{\tt hardouin@math.univ-toulouse.fr}}}

\date{followed by the appendix\\~\\
{\Large ``The Galois $D$-groupoid of a $q$-difference system''}\\~\\
by Anne Granier\thanks{Anne Granier, Institut de Math\'{e}matiques de Toulouse,
118 route de Narbonne,
31062 Toulouse Cedex 9, France.
{\tt anne.granier@gmail.com}}}

\title
{Parameterized generic Galois group for $q$-difference equations
\footnotetext{Date: \today}
\footnotetext{Work partially supported by ANR-06-JCJC-0028.}}

\begin{document}
\bibliographystyle{alpha}

\maketitle


\begin{abstract}
We introduce the parameterized generic Galois group of a $q$-difference module,
that is a differential group in the sense of Kolchin.
It is associated to
the smallest \emph{differential} tannakian category generated by the $q$-difference module, equipped with the forgetful functor.
The main result of \cite{diviziohardouinqGroth} leads to an adelic description of the parameterized
generic Galois group, in the spirit of the Grothendieck-Katz's conjecture on $p$-curvatures.
Using this description, we show that the Malgrange-Granier $D$-groupoid of a
nonlinear $q$-difference
system coincides, in the linear case, with the parameterized generic Galois group introduced here.
The paper is followed by an appendix by A. Granier, that provides a quick introduction to the $D$-groupoid of a
non-linear $q$-difference equation.
\par
The present paper is the natural continuation of \cite{diviziohardouinqGroth}. However it can be read independently.
\end{abstract}

\setcounter{tocdepth}{1}
\tableofcontents

\section*{Introduction}

In this paper, we combine the Grothendieck-Katz conjecture, proved in \cite{diviziohardouinqGroth},
with the parameterized Galois theory of difference equations
of Hardouin and Singer, to obtain an adelic approach to the study of the differential
algebraic relations satisfied by the solutions of a $q$-difference equation.
This allows us to prove that one can recover the Galois theory of linear
$q$-difference equations from the Malgrange-Granier $D$-groupoid, defined for nonlinear $q$-difference equations.
In the process, we build the first path between Kolchin's theory of linear differential groups
and Malgrange's $D$-groupoid, answering a question of B. Malgrange (see \cite[page 2]{Malgrangepseudogroupes}).
\par
In \cite{HardouinSinger}, the authors attach to a $q$-difference equation
a parameterized Galois group, which is a differential group \emph{\`a la Kolchin}. This
is a subgroup of the group of invertible matrices of a given order,
defined by a set of nonlinear algebraic differential equations.
The differential dimension of this Galois group measures
the hypertranscendence properties of a basis of solutions. We recall that
a function $f$ is hypertranscendental over a field $F$ equipped with a derivation $\partial$
if $F[\partial^n(f), n\geq 0]/F$ is a transcendental extension of infinite degree, or equivalently,
if $f$ is not a solution of a nonlinear algebraic differential equation with coefficients in
$F$. The question of hypertranscendence of solutions of functional equations appears in various mathematical domains: in special function theory (see for instance \cite{LiYezeta}, \cite{Mar} for the differential independence $\zeta$ and $\Gamma$ functions), in enumerative combinatorics  (see for instance \cite{BouPetwalk} for
problems of hypertranscendence and $D$-finiteness\footnote{A function $f$ is $D$-finite over a differential field $(\cF,\partial)$ if it is solution of a linear differential
equation with coefficients in $\cF$.} specifically related to $q$-difference equations\footnote{In \cite{BouPetwalk}, the authors consider
some formal power series generated by enumeration of random walks with constraints. Such generating series are solutions of $q$-difference equations: one natural step towards
their rationality is to establish whether they satisfy an algebraic (maybe nonlinear) differential equation. In fact, as proven by J.-P. Ramis, a formal power series
which is solution of a linear differential equation and a linear $q$-difference equation, both with coefficients in $\C(x)$, is necessarily rational (see \cite{RamisToulouse}).
Other examples of $q$-difference equations for which it would be interesting to establish hypertranscendency are given in \cite{Bousquetmelou95} and
\cite{Bousquetmelou96}.}),
\dots
The problem of the parameterized Galois group of Hardouin-Singer is that it is defined over the
differential closure of the elliptic functions over $C^*/q^\Z$, which is quite a big field.
\par
Here we introduce a parameterized generic Galois group attached to a $q$-difference
module $\cM=(M,\Sgq)$ over a field of rational functions $K(x)$, with coefficients in a field $K$ of characteristic $0$, in the spirit
of \cite{Katzbull}, \cite{diviziohardouin} and \cite{diviziohardouinqGroth}. This is a differential subgroup of $GL(M)$, defined over $K(x)$.
According to the nature of $K$ and $q$ we define a set of places $\cC$ of $K$ and a set of elements $\phi_v$ in $K$ so that we can prove
that \emph{the parameterized generic Galois group is the smallest differential subgroup of $GL(M)$
whose reduction modulo $\phi_v$ contains the $v$-curvatures modulo $\phi_v$, for almost all $v\in \cC$} (see \S\ref{subsec:car0} for precise definitions and statements).
\par
In this way, we have replaced the group introduced in \cite{HardouinSinger} by
another group defined over a smaller field, namely $K(x)$, and which admits an arithmetic characterization.
Notice that the differential generic Galois group is Zariski dense in the generic Galois group considered in \cite{diviziohardouinqGroth}.
The comparison among the parameterized generic Galois group and the parameterized Galois group of \cite{HardouinSinger} is studied in
\cite{diviziohardouinComp}, where we clarify the links between the many different Galois theories for $q$-difference equations in the literature.
\par
Malgrange has defined and studied Galois $D$-groupoids for nonlinear differential equations.
In the differential case the Galois $D$-groupoid has been shown to generalize the Galois group
in the sense of Kolchin-Picard-Vessiot by Malgrange himself (\cf \cite{MalgGGF}).
Roughly speaking the Galois $D$-groupoid is a groupoid of local diffeomorphisms of a variety
defined by a sheaf of differential ideals in the jet space of the variety.
This construction can be generalized to quite general dynamical systems. In the particular case of nonlinear $q$-difference equations
it has been done by A. Granier.
The idea of considering a groupoid defined by a differential structure is actually quite natural.
In fact, it encodes the ``linearizations'' of the dynamical system along its orbits, \ie the variational equations attached to the dynamical system, also called the linearized equations.
One of the two proofs (\cf \cite{casaleroques}) of the analog of Morales-Ramis theorem for $q$-difference equations, \ie the
connection between integrability of a nonlinear system and solvability of the Lie algebra of its Galois $D$-groupoid, was done
under the following conjecture:
``for linear (q-)difference systems, the action of Malgrange
groupoid on the fibers gives the classical Galois groups'' (\cf \cite[\S7.3]{casaleroques}).
\par
As an application of the characterization of the parameterized generic Galois group explained above,
we actually show that the $D$-groupoid of
a dynamical system associated with a nonlinear $q$-difference equation, as introduced by A. Granier,
generalizes the notion
of parameterized generic Galois group. In the particular case of linear $q$-difference systems with constant coefficients
we retrieve an algebraic Galois group and the result
obtained in \cite{GranierFourier}. In other words, for linear $q$-difference systems,
the Galois $D$-groupoid essentially coincides with the differential generic Galois group, so it contains
more informations than the Picard-Vessiot
Galois group. In fact, thanks to the comparison results in the \cite{diviziohardouinComp} (see also \cite{diviziohardouin}), we conclude that the
Galois $D$-groupoid of a linear $q$-difference system allows to recover both the classical, \ie Picard-Vessiot,
Galois group (\cf \cite{vdPutSingerDifference}, \cite{SauloyENS}) by taking
its Zariski closure and extending the base field conveniently, and the Hardouin-Singer parameterized Galois group
(\cf \cite{HardouinSinger}) only by extending the field. This proves a more refined version of the conjecture in \cite{casaleroques}.
It moreover gives a first answer to a question asked by Malgrange, see \cite[page 2]{Malgrangepseudogroupes}, on the link between
his $D$-groupoid and the Kolchin differential groups appearing in the parameterized Galois theory.

\subsubsection*{Acknowledgements.}
We would like to thank
D. Bertrand, Z. Djadli, C. Favre, M. Florence, A. Granier, D. Harari, F. Heiderich, A. Ovchinnikov, B. Malgrange, J-P. Ramis, J. Sauloy, M. Singer, J. Tapia, H. Umemura and
M. Vaquie for the many discussions on different points of this paper.
\par
We would like to thank the ANR projet Diophante that has made possible a few reciprocal visits,
and the Centre International de Rencontres
Math\'{e}matiques in Luminy for supporting us \emph{via} the Research in pairs program and
for providing a nice working atmosphere.


\section{Parameterized generic Galois groups}
\label{sec:genericgaloisgroup}

Let $K$ be a field of characteristic zero and $K(x)$ be the field of rational functions
in $x$ with coefficients in $K$.
The field $K(x)$ is naturally a $q$-difference algebra for any $q\in K\smallsetminus\{0,1\}$, \ie is equipped with the
operator
$$
\begin{array}{rccc}
\sgq:&K(x)&\longrightarrow&K(x)\\
&f(x)&\longmapsto&f(qx)
\end{array}.
$$
More generally, we will consider a field $K$, with a fixed element $q\neq 0,1$, and
an extension $\cF$ of $K(x)$ equipped with a $q$-difference operator, still called $\sgq$, extending
the action of $\sigma_q$.

\begin{defn}
A \emph{$q$-difference module over $\cF$ (of rank $\nu$)}
is an $\cF$-vector space $M_{\cF}$ (of finite dimension $\nu$)
equipped with an invertible $\sgq$-semilinear operator, \ie
$$
\Sgq(fm)=\sgq(f)\Sgq(m),
\hbox{~for any $f\in \cF$ and $m\in M_{\cF}$.}
$$
A \emph{morphism of $q$-difference modules over $\cF$} is a morphism
of $\cF$-vector spaces, commuting with the $q$-difference structures
(for more generalities on the topic, \cf \cite{vdPutSingerDifference},
\cite[Part I]{DVInv}
or \cite{gazette}).
We denote by $Diff(\cF, \sgq)$ the category of $q$-difference modules over $\cF$.
\end{defn}

Let $\cM_{\cF}=(M_{\cF},\Sgq)$ be a $q$-difference module over $\cF$ of rank $\nu$.
We fix a basis $\ul e$ of $M_{\cF}$ over $\cF$ and we set:
$$
\Sgq\ul e=\ul e A,
$$
with $A\in \GL_\nu(\cF)$.
A horizontal vector $\vec y\in \cF^\nu$ with respect to the basis $\ul e$ for
the operator $\Sgq$ is a vector
that verifies $\Sgq(\ul e\vec y)=\ul e\vec y$, \ie $\vec y=A\sgq(\vec y)$. Therefore we call
$$
\sgq(Y)=A_1Y,
\hbox{~with~}A_1=A^{-1},
$$
th\emph{e ($q$-difference) system associated to $\cM_{\cF}$ with respect to the basis $\ul e$}.
Recursively, we obtain a family of higher order $q$-difference systems:
$$
\sgq^n(Y)=A_n Y
\hbox{~with $A_n \in \GL_\nu(\cF)$.}
$$
Notice that $A_{n+1}= \sgq(A_n)A_1$.
It is convenient to set $A_0=1$.

\medskip
Let $\cF$ be a $q$-difference-differential field of zero characteristic,
that is, an extension of $K(x)$ equipped with an extension of the $q$-difference operator $\sgq$ and
a derivation $\partial$ commuting with $\sgq$ (\cf \cite[\S1.2]{hardouincompositio}).
For instance, the $q$-difference-differential field $(K(x),\sgq,x\frac{d}{dx})$ satisfies these assumptions.
\par
We can define an action of the derivation $\partial$ on the category
$Diff(\cF,\sgq)$, twisting the $q$-difference modules with the right $\cF$-module $\cF[\partial]_{\leq 1}$
of differential operators of order less or equal
than one.
We recall that the structure of right $\cF$-module on $\cF[\partial]_{\leq 1}$
is defined via the Leibniz rule, \ie
$$
\partial \lambda =\lambda \partial + \partial(\lambda),
\hbox{~for any $\la\in\cF$.}
$$
Let $V$ be an $\cF$-vector space. We denote by $F_\partial(V)$ the tensor product of
the right $\cF$-module $\cF[\partial]_{\leq 1}$ by the left $\cF$-module $V$:
$$
F_\partial(V):=\cF[\partial]_{\leq 1} \otimes_\cF V.
$$
We will write simply $v$ for $1\otimes v\in F_\partial(V)$ and $\partial(v)$ for
$\partial\otimes v\in F_\partial(V)$, so that $av+b\partial(v):=(a+b\partial)\otimes v$, for any $v\in V$ and
$a+b\partial\in \cF[\partial]_{\leq 1}$.
Similarly to the constructions of
\cite[Proposition 16]{grangmaiso} for $\cD$-modules, we endowe $F_\partial(V)$
with a left $\cF$-module structure such that if $\lambda \in\cF$:
$$
\lambda\partial(v)=\partial(\lambda v)-\partial(\lambda)v,
\ \mbox{for all} \, v \in V.
$$
This construction comes out of the Leibniz rule
$ \partial( \lambda v)= \lambda \partial(v)+\partial(\lambda)v$,
which justifies the notation introduced above.

\begin{defn}\label{defn:prolong}
The \emph{prolongation functor $F_\partial$} is defined on the category of $\cF$-vector spaces as follows.
It associates to any object $V$ the $\cF$-vector space $F_\partial(V)$. If $f:V\longrightarrow W$ is a
morphism of $\cF$-vector space then we define
$$
F_\partial(f):F_\partial(V)\rightarrow F_\partial(W),
$$
setting $F_\partial(f)(\partial^k(v))=\partial^k(f(v))$, for any $k=0,1$ and any $v\in V$.
\par
The prolongation functor $F_\partial$ restricts to a functor from the category $Diff(\cF,\sgq)$ to itself in the
following way:
\begin{enumerate}

\item
If $\cM_\cF:=(M_\cF, \Sgq)$ is an object of $Diff(\cF,\sgq)$ then
$F_\partial(\cM_\cF)$
is the $q$-difference module, whose underlying $\cF$-vector space is
$F_\partial(M_\cF)=\cF[\partial]_{\leq 1} \otimes M_\cF$, as above, equipped with the $q$-invertible
$\sgq$-semilinear operator defined by $\Sgq(\partial^k(m)):=\partial^k(\Sgq(m))$ for $k=0,1$.

\item If $f \in Hom(\cM_\cF,\cN_\cF)$ then $F_\partial(f)$ is defined in the same way as for
$\cF$-vector spaces.
\end{enumerate}
\end{defn}

\begin{rmk}\label{rmk:prolongationmatrix}
This formal definition comes from a simple and concrete idea.
Let $\cM_\cF$ be an object of $ Diff(\cF,\sgq)$. We fix a basis
$\ul e$ of $\cM_\cF$ over $\cF$ such that $\Sgq\ul e=\ul e A$.
Then $(\ul e,\partial(\ul e))$ is a basis of $F_\partial(M_\cF)$ and
$$
\Sgq (\ul e,\partial(\ul e))=(\ul e,\partial(\ul e))
\begin{pmatrix}A & \partial A \\0 & A\end{pmatrix}.
$$
In other terms, if
$\sgq(Y)=A^{-1}Y$ is a $q$-difference system associated to $\cM_\cF$
with respect to a fixed basis $\ul e$,
the $q$-difference system associated to $F_\partial(\cM_\cF)$ with respect to the basis
$\ul e,\partial(\ul e)$ is:
$$
\sgq(Z)=\begin{pmatrix}A^{-1} & \partial(A^{-1}) \\0 & A^{-1}\end{pmatrix}Z
=\begin{pmatrix}A & \partial A \\0 & A\end{pmatrix}^{-1}Z.
$$
If $Y$ is a solution of $\sgq(Y)=A^{-1}Y$ in some $q$-difference-differential extension of $\cF$ then
we have:
$$
\sgq\begin{pmatrix}\partial Y\\ Y\end{pmatrix}
=
\begin{pmatrix}A^{-1} & \partial(A^{-1}) \\0 & A^{-1}\end{pmatrix}
\begin{pmatrix}\partial Y\\ Y\end{pmatrix},
$$
in fact the commutation of $\sgq$ and $\partial$ implies:
$$
\sgq(\partial Y)=\partial(\sgq Y)=\partial(A^{-1}\,Y)=
A^{-1}\,\partial Y+\partial(A^{-1})\,Y.
$$
\end{rmk}

Let $V$ be a finite dimensional $\cF$-vector space. We denote by $Constr_{\cF}^{\partial} (V)$ the smallest family of
finite dimensional $\cF$-vector spaces containing $V$ and closed with respect to the constructions of linear algebra
(\ie direct sums, tensor product, symmetric and antisymmetric product, dual. See \cite[\S4]{diviziohardouinqGroth}) and the functor $F_\partial$.
We will say that an element $Constr_{\cF}^{\partial} (V)$ is a \emph{construction of differential linear algebra} of $V$.
By functoriality, the linear algebraic group $\GL(V)$ operates on $Constr_{\cF}^{\partial}(V)$.
For example $g \in \GL(V)$ acts on $F_\partial(V)$
through $g(\partial^k(v))=\partial^k(g(v)), k=0,1$.
\par
If we start with a $q$-difference module
$\cM_\cF=(M_\cF,\Sgq)$ over $\cF$, then every object of $Constr_{\cF}^{\partial}(M_\cF)$
has a natural structure of $q$-difference module (see also \cite[\S4]{diviziohardouinqGroth}).
We will denote $Constr_{\cF}^{\partial}(\cM_\cF)$
the family of $q$-difference modules obtained in this way.

\begin{defn} \label{defn:diffgengal}
We call \emph{parameterized generic Galois group}
of an object $\cM_\cF=(M_\cF,\Sgq)$ of $Diff(\cF, \sgq)$ the
group defined by
$$
\begin{array}{l}
\ds
Gal^{\partial}(\cM_\cF,\eta_{\cF}):=
\Big\{g\in \GL(M_{\cF}):
g(N)\subset N\,\hbox{~for all sub-$q$-difference module}\\
\ds
\hbox{ $(N,\Sgq)$ contained in an object of $Constr_{\cF}^{\partial}(\cM_\cF)$}\Big\}\subset \GL(M_{\cF}).
\end{array}
$$
\end{defn}

Following \cite[Definition page 3057]{DifftanOv}, we denote by $\langle\cM_\cF\rangle^{\otimes,\partial}$
the full abelian tensor subcategory of $Diff(\cF,\sgq)$ generated by $\cM_\cF$
and closed under the prolongation functor $F_\partial$.
Then $\langle\cM_\cF\rangle^{\otimes,\partial}$
is naturally a differential tannakian category, when
equipped with the forgetful functor
$$
\eta_\cF:\langle\cM_{K(x)}\rangle^{\otimes,\partial}\longrightarrow\{\hbox{$\cF$-vector spaces}\}.
$$
The functor $Aut^\otimes(\eta_\cF)$ defined over the category of
$\cF$-algebras is representable by the differential group (scheme) $Gal^\partial(\cM_\cF,\eta_\cF)$.

\medskip
For further reference, we recall (a particular case of) the Ritt-Raudenbush theorem (\cf \cite[Theorem 7.1]{Kapldiffalg}):

\begin{thm}\label{thm:rittraudenbush}
Let $(\cF,\partial)$ be a differential field of zero characteristic.
If $R$ is a finitely generated $\cF$-algebra equipped with a derivation $\partial$,
extending the derivation $\partial$ of $\cF$, then $R$ is $\partial$-noetherian.
\end{thm}

This means that any ascending chain of \emph{radical} differential ideals (\ie radical $\partial$-stable ideals) is
stationary or equivalently that every radical $\partial$-ideal is $\partial$-finitely generated
(which in general does not mean that it is a finitely generated ideal).
Theorem \ref{thm:rittraudenbush} asserts that $\GL_\nu(\cF)$ is a $\partial$-noetherian differential variety in the sense
that its algebra of differential rational functions $\cF\l\{Y,\frac{1}{\det Y}\r\}_\partial$ is $\partial$-noetherian.
We recall that this is an algebra defined as follows.
We denote by $\cF\{Y\}_{\partial}$ the ring of differential
polynomials in the $\partial$-differential indeterminates
$Y=\{y_{i,j}:i,j=1,\dots,\nu\}$.
This means that $\cF \{Y\}_{\partial}$ is isomorphic as a differential $\cF$-algebra to a polynomial ring in infinite indeterminates
$\cF[\wtilde y_{i,j}^k;i,j=1,\dots,\nu,\, k\geq 0]$, equipped with a derivation $\partial$ extending the derivation of $\cF$
and such that $\partial \wtilde y_{i,j}^k=\wtilde y_{i,j}^{k+1}$, via the map
$$
\begin{array}{ccc}
\cF\{Y\}_\partial &\longrightarrow& \cF[\wtilde y_{i,j}^k;i,j=1,\dots,\nu,\, k\geq 0]\\
\partial^k(y_{i,j})&\longmapsto &\wtilde y_{i,j}^k
\end{array}.
$$
The differential Hopf-algebra $\cF\l\{ Y,\frac{1}{\det Y}\r\}_{\partial}$
of $\GL_{\nu}(\cF)$ is obtained from $\cF\{Y\}_{\partial}$ by inverting $\det Y$.

\begin{prop}
The group $Gal^{\partial}(\cM_\cF,\eta_{\cF})$ is a linear
differential $\cF$-subgroup of $\GL(M_{\cF})$.
\end{prop}

\begin{proof}
Let $\cM_\cF=(M_\cF,\Sgq)$ be an object of $Diff(\cF,\sgq)$.
Following \cite[Section 2]{Tandifgrov}, we look at linear differential groups defined over $\cF$ as representable functors from the
category of $\partial$-$\cF$-differential algebras, \ie
commutative associative $\cF$-algebras $A$ with unit, equipped with a derivation
$\partial : A \rightarrow A$ extending the one of $\cF$, to the category of groups.
Now, the functor $\ul{Stab}$ that associates to a $\partial$-$\cF$-differential
algebra $A$, the stabilizer, inside $\GL(M_\cF)(A)$, of $N_\cF\otimes A$
for all sub-$q$-difference module $\cN_\cF=(N_\cF,\Sgq)$
contained in an object of $Constr_{\cF}^{\partial}(\cM_\cF)$, is representable by a
linear differential group, which is exactly $Gal^\partial(\cM_\cF,\eta_\cF)$.
It is the differential analogous of \cite[II.1.36]{demazuregabriel}.
\end{proof}




The noetherianity implies the following:

\begin{cor}
The parameterized generic Galois group $Gal^{\partial}(\cM_\cF,\eta_{\cF})$ can
be defined as the differential stabilizer of a line in a construction of differential algebra of $\cM_{\cF}$.
This line can be chosen so that it is also a $q$-difference module in the category $\langle\cM_\cF\rangle^{\otimes,\partial}$.
\end{cor}

\begin{proof}
Let $\l\{\cW^{(i)};i\in I_h\r\}_h$ is an ascending chain of
finite sets of $q$-difference submodules contained in some elements of
$Constr^\partial(\cM_{K(x)})$ and let
$\cG_h$ be the differential subgroup of $\GL(M_\cF)$ defined as the stabilizer of $\{\cW^{(i)};i\in I_h\}$.
Then the ascending chain of the
radical differential ideals of the differential rational functions that annihilates
$\cG_h$ is stationary and so does the chain of differential groups $\cG_h$.
This proves that $Gal^{\partial}(\cM_\cF,\eta_{\cF})$ is the stabilizer of a finite number
of $q$-difference submodules $\cW^{(i)}$, $i\in I$,
contained in some elements of $Constr^\partial(\cM_{K(x)})$.
It follows from a standard noeterianity argument that
$Gal^{\partial}(\cM_\cF,\eta_{\cF})$ is the stabilizer of the maximal exterior
power of the direct sum of the $\cW^{(i)}$'s (see \cite[Remark 4.2]{diviziohardouinqGroth}).
\end{proof}

Let
$Gal(\cM_\cF,\eta_{\cF})$ be the generic Galois group defined in \cite[Definition 4.1]{diviziohardouinqGroth}.
We have the following inclusion, that we will characterize in a more
precise way in the next pages:

\begin{lemma}\label{lemma:inclgroup}
Let $\cM_\cF$ be an object of $Diff(\cF,\sgq)$.
The following inclusion of differential groups holds
$$Gal^{\partial}(\cM_\cF,\eta_{\cF}) \subset Gal(\cM_\cF,\eta_{\cF}).$$
\end{lemma}

\begin{rmk}
We would like to put the accent on the fact that differential groups are not algebraic groups,
while algebraic groups are differential groups (whose ``equations'' do not contain ``derivatives'').
In particular, the parameterized generic Galois group is not an algebraic subgroup of the generic Galois group
but only an differential subgroup. Later, for $\cF=K(x)$, we will prove that
$Gal^{\partial}(\cM_\cF,\eta_{\cF})$ is actually Zariski dense in $Gal(\cM_\cF,\eta_{\cF})$.
\end{rmk}

\begin{proof}
We recall, that the algebraic group $Gal(\cM_\cF,\eta_{\cF})$ is defined as the stabilizer in $\GL( M_{\cF})$
of all the subobjects contained in a construction of linear algebra of $\cM_\cF$ .
Because the list of subobjects contained in a construction of differential linear algebra of $\cM_\cF$ includes
those contained in a construction of linear algebra of $\cM_{\cF}$,
we get the claimed inclusion.
\end{proof}

\section{Characterization of the parameterized generic Galois group by curvatures}
\label{subsec:car0}

From now on we focus on the special case $\cF= K(x)$, where $K$ is a finitely generated extension of $\Q$.
The last assumption is not restrictive for the sequel, since we will always be able to reduce to this case.
We endow $K(x)$ with the derivation $\partial:= x \frac{d}{dx}$, that commutes with $\sgq$.
In this section, we are going to deduce a characterization of
$Gal^{\partial}(\cM_{K(x)},\eta_{K(x)})$ from the $q$-analogue of Grothendieck-Katz conjecture
(see \cite[\S5]{diviziohardouinqGroth}).

\medskip
Let $(\cM=(M_{K(x)},\Sgq)$ be a $q$-difference module over $K(x)$.
We need the following notation (\cf \cite[\S5]{diviziohardouinqGroth}):
\begin{trivlist}

\item $\bullet$
If \emph{$q$ is algebraic over $\Q$, but not a root of unity}, we are in the following situation.
We call $Q$ the algebraic closure of $\Q$ inside $K$,
$\cO_Q$ the ring of integer of $Q$, $\cC$ the set of finite places $v$ of $Q$ and $\pi_v$ a $v$-adic uniformizer.
For almost all $v\in\cC$ the following are well defined:
the order $\kappa_v$, as a root of unity, of the reduction of $q$ modulo $\pi_v$
and the positive integer power $\phi_v$ of
$\pi_v$, such that $\phi_v^{-1}(1-q^{\kappa_v})$
is a unit of $\cO_Q$.
The field $K$ has the form $Q(\ul a,b)$, where $\ul a=(a_1,\dots,a_r)$ is a transcendent basis
of $K/Q$ and $b$ is a primitive element of the algebraic extension $K/Q(\ul a)$.
Choosing conveniently the set of generators $\ul a,b$ and $P(x)\in \cO_Q\l[\ul a,b,x\r]$,
we can always find an algebra $\cA$ of the form
\beq
\cA=\cO_Q\l[\ul a,b,x,\frac{1}{P(x)},\frac{1}{P(qx)},...\r]
\eeq
and a $\Sgq$-stable $\cA$-lattice $M$ of $\cM_{K(x)}$, so that
we can consider the
$\cA/(\phi_v)$-linear operator
$$
\Sgq^{\kappa_v}:M\otimes_{\cA}\cA/(\phi_v)\longrightarrow M\otimes_{\cA}\cA/(\phi_v),
$$
that we will call the \emph{$v$-curvature of $\cM_{K(x)}$-modulo $\phi_v$}.
Notice that $\cO_Q/(\phi_v)$ is not an integral domain in general.

\item $\bullet$
If \emph{$q$ is transcendental over $\Q$}, then there exists a subfield $k$ of $K$ such that
$K$ is a finite extension of $k(q)$. We denote by
$\cC$ the set of places of $K$ that extend the places of $k(q)$, associated to
irreducible polynomials
$\phi_v$ of $k[q]$, that vanish at roots of unity.
Let $\kappa_v$ be the order of the roots of $\phi_v$, as roots of unity.
Let $\cO_K$ be the integral closure of $k[q]$ in $K$.
Choosing conveniently $P(x)\in \cO_K[x]$,
we can always find an algebra $\cA$ of the form:
\beq
\cA=\cO_K\l[x,\frac{1}{P(x)},\frac{1}{P(qx)},...\r]
\eeq
and a $\Sgq$-stable $\cA$-lattice $M$ of $\cM_{K(x)}$, so that
we can consider the
$\cA/(\phi_v)$-linear operator
$$
\Sgq^{\kappa_v}:M\otimes_{\cA}\cA/(\phi_v)\longrightarrow M\otimes_{\cA}\cA/(\phi_v),
$$
that we will call the \emph{$v$-curvature of $\cM_{K(x)}$-modulo $\phi_v$}.
Notice that, once again, $\cO_K/(\phi_v)$ is not an integral domain in general.

\item $\bullet$
If \emph{$q$ it a primitive root of unity of order $\kappa$},  we define $\cC$ to be the set containing
only the trivial valuation $v$ on $K$, $\phi_v=0$ and $\kappa_v=\kappa$. Then there exists a polynomial
$P(x)\in K[x]$ such the algebra $\cA=K\l[x,\frac{1}{P(x)},\frac{1}{P(qx)},...\r]$ is $\sgq$-stable and there exists
a $\Sgq$-stable $\cA$-lattice $M$ of $\cM_{K(x)}$, so that
we can consider the
$\cA/(\phi_v)$-linear operator
$$
\Sgq^{\kappa_v}:M\otimes_{\cA}\cA/(\phi_v)\longrightarrow M\otimes_{\cA}\cA/(\phi_v),
$$
that we will call the \emph{$v$-curvature of $\cM_{K(x)}$-modulo $\phi_v$}.
Notice that this is simple the $\kappa$-th iterate of $\Sgq$, namely $\Sgq^\kappa:M\longrightarrow M$.
\end{trivlist}

We recall that  $\cM_{K(x)}=(M_{K(x)},\Sgq)$ is trivial if and only if there exists a basis $\ul e$ of $M_{K(x)}$
over $K(x)$ such that $\Sgq\ul e=\ul e$. This is equivalent to ask that any $q$-difference systems associated
to $\cM_{K(x)}$ (with respect to any basis) has a fundamental solutions in $GL_\nu(K(x))$.
Then the main result of \cite{diviziohardouinqGroth} states:

\begin{thm}[{\cf \cite[Theorem 5.4]{diviziohardouinqGroth}}]\label{thm:GrothKatz}
A $q$-difference module $\cM_{K(x)}=(M_{K(x)},\Sgq)$ over $K(x)$ is trivial if and only if there exists
an algebra $\cA$, as above, and a $\Sgq$-stable $\cA$-lattice $M$ of $M_{K(x)}$ such that the map
$$
\Sgq^{\kappa_v}:M\otimes_{\cA}\cA/(\phi_v)\longrightarrow M\otimes_{\cA}\cA/(\phi_v),
$$
is the identity, for any $v$ in a cofinite nonempty subset of $\cC$.
\end{thm}

\medskip
Notice that the algebra $\cA$ (in each case considered above) is stable under the action of the derivation $\partial$.
Let $\cM_{K(x)}=(M_{K(x)},\Sgq)$ be a $q$-difference module.
The differential version of Chevalley's theorem (\cf \cite[Proposition 14]{cassdiffgr}, \cite[Theorem 5.1]{ChevOv}) implies that
any closed differential subgroup $G$ of $\GL(M_{K(x)})$
can be defined as the stabilizer of some line
$L_{K(x)}$ contained in an object $\cW_{K(x)}$ of $\langle\cM_{K(x)}\rangle^{\otimes, \partial}$.
Because the derivation does not modify the set of poles of a rational function,
the lattice $\cM$ of $\cM_{K(x)}$ determines a $\Sgq$-stable $\cA$-lattice
of all the objects of $\langle\cM_{K(x)} \rangle^{\otimes, \partial}$.
In particular, the $\cA$-lattice $M$ of $M_{K(x)}$ determines an $\cA$-lattice $L$ of $L_{K(x)}$ and
an $\cA$-lattice $W$ of $W_{K(x)}$. The latter is the underlying space of a $q$-difference module
$\cW=(W,\Sgq)$ over $\cA$.

\begin{defn}
Let $\wtilde\cC$ be a nonempty cofinite subset of $\cC$ and $(\La_v)_{v\in\wtilde\cC}$
be a family of $\cA/(\phi_v)$-linear operators acting on $M\otimes_{\cA}\cA/(\phi_v)$.
We say that \emph{the differential group $G$ contains the operators $\La_v$ modulo $\phi_v$
for almost all $v\in\cC$}
if for almost all (and at least one) $v\in\wtilde\cC$ the operator $\La_v$ stabilizes
$L\otimes_{\cA}\cA/(\phi_v)$ inside $W\otimes_{\cA}\cA/(\phi_v)$:
$$
\La_v\in Stab_{\cA/(\phi_v)}(L\otimes_\cA\cA/(\phi_v)).
$$
\end{defn}

\begin{rmk}
The differential
Chevalley's theorem and the $\partial$-noetherianity of $\GL(M_{K(x)})$ imply that
the notion of a differential group containing the operators $\La_v$ modulo $\phi_v$ for almost all $v\in\cC$ and the
smallest closed differential subgroup of $\GL(M_{K(x)})$ containing
the operators $\La_v$ modulo $\phi_v$ for almost all $v\in\cC$ are well defined.
In particular they are independent of the choice of $\cA$, $\cM$ and $L_{K(x)}$
(See \cite[10.1.2]{DVInv}, \cite[Remark 4.4]{diviziohardouinqGroth}).
\end{rmk}

The main result of this section is the following:

\begin{thm}\label{thm:diffgenGalois}
The differential group $Gal^{\partial}(\cM_{K(x)},\eta_{K(x)})$ is the smallest closed differential subgroup of
$\GL(M_{K(x)})$ that contains the operators $\Sgq^{\kappa_v}$ modulo
$\phi_v$, for almost all $v\in\cC$.
\end{thm}

\begin{proof}
The lemmas below plus the differential Chevalley theorem
allow to prove Theorem \ref{thm:diffgenGalois} in exactly the same way as
\cite[Theorem 4.5]{diviziohardouinqGroth}.
\end{proof}

\begin{lemma}\label{lemma:diffgenGalois}
The differential group $Gal^{\partial}(\cM_{K(x)},\eta_{K(x)})$ contains the operators
$\Sgq^{\kappa_v}$ modulo $\phi_v$ for almost all $v\in\cC$.
\end{lemma}

\begin{proof}
The statement follows immediately from the fact that $Gal^{\partial}(\cM_{K(x)},\eta_{K(x)})$
can be defined as the stabilizer of one rank one $q$-difference module
in $\langle\cM_{K(x)}\rangle^{\otimes,\partial}$, which is \emph{a fortiori} stable under the action
of $\Sgq^{\kappa_v}$.
\end{proof}

\begin{lemma}
$Gal^{\partial}(\cM_{K(x)},\eta_{K(x)})=\{1\}$ if and only if
$\cM_{K(x)}$ is a trivial $q$-difference module.
\end{lemma}

\begin{proof}
The proof is analogous to the proof of \cite[Corollary 4.8]{diviziohardouinqGroth}.
It suffices to replace $\langle\cM_{K(x)}\rangle^{\otimes}$ with $\langle\cM_{K(x)}\rangle^{\otimes,\partial}$.
\end{proof}

We obtain the following:

\begin{cor}\label{cor:diffdens}
The parameterized generic Galois group $Gal^{\partial}(\cM_{K(x)},\eta_{K(x)})$ is a Zariski dense subset of
the algebraic generic Galois group $Gal(\cM_{K(x)},\eta_{K(x)})$.
\end{cor}

\begin{proof}
We have seen in Lemma \ref{lemma:inclgroup} that $Gal^{\partial}(\cM_{K(x)},\eta_{K(x)})$ is a subgroup
of $Gal(\cM_{K(x)},\eta_{K(x)})$. By Theorem \ref{thm:diffgenGalois} (resp. \cite[Theorem 4.8]{diviziohardouinqGroth})
we have that
the generic Galois group $Gal^\partial(\cM_{K(x)},\eta_{K(x)})$ (resp. $Gal(\cM_{K(x)},\eta_{K(x)})$)
is the smallest closed differential subgroup (resp. closed algebraic group) of
$\GL(M_{K(x)})$ that contains the operators $\Sgq^{\kappa_v}$ modulo
$\phi_v$, for almost all $v\in\cC$. This immediately implies the Zariski density.
\end{proof}

\begin{exa}
The logarithm is solution both a $q$-difference and a differential system:
$$
Y(qx)=\begin{pmatrix}
1&\log q\\ 0&1
\end{pmatrix}Y(x),
\hskip 15 pt
\partial Y(x)=\begin{pmatrix}
0&1\\ 0&0
\end{pmatrix}Y(x).
$$
It is easy to verify that the two systems are integrable in the sense that
$\partial\sgq Y(x)=\sgq\partial Y(x)$ (and therefore that the induced condition on the matrices of the systems is verified).
\par
By iterating the $q$-difference system for any $n\in\Z_{>0}$ we obtain:
$$
Y(q^nx)=\begin{pmatrix}
1&n\log q\\ 0&1
\end{pmatrix}Y(x).
$$
This implies that the parameterized generic Galois group is the subgroup of $\mathbb G_{a,K(x)}$
defined by the equation $\partial y=0$. This coincides with the group $\mathbb G_{a,K}$, compatibly
with the integrability criteria in \cite{HardouinSinger}.
For more precise comparison results with the theory developed in \cite{HardouinSinger}, we refer
to \cite{diviziohardouinComp}.
\end{exa}

We conclude with some remarks on complex $q$-difference modules.
 Let
$\cM_{\C(x)}=(M_{\C(x)},\Sgq)$ be a $q$-difference module over
$\C(x)$. We can consider a finitely generated extension of $K$ of
$\Q$ such that there exists a $q$-difference module
$\cM_{K(x)}=(M_{K(x)},\Sgq)$ satisfying
$\cM_{\C(x)}=\cM_{K(x)}\otimes_{K(x)}\C(x)$. We can of course
define, as above, two parameterized generic Galois groups,
$Gal^\partial(\cM_{K(x)},\eta_{K(x)})$ and
$Gal^\partial(\cM_{\C(x)},\eta_{\C(x)})$.
A (differential)
noetherianity argument, that we have already used several times, on
the submodules stabilized by those groups shows the following:

\begin{prop}\label{prop:finitegenextension}
In the notation above we have:
$$
Gal^\partial(\cM_{\C(x)},\eta_{\C(x)})
\subset Gal^\partial(\cM_{K(x)},\eta_{K(x)})\otimes_{K(x)}\C(x).
$$
Moreover there exists a finitely generated extension $K^\p$ of $K$
such that
$$
Gal^\partial(\cM_{K(x)}\otimes_{K(x)}{K^{\p}(x)},\eta_{K^{\p}(x)})\otimes_{K^{\p}(x)}\C(x)\cong
Gal^\partial(\cM_{\C(x)},\eta_{\C(x)}).
$$
\end{prop}

We can informally rephrase Theorem \ref{thm:diffgenGalois} in the
following way:

\begin{thm}\label{thm:complexmodulesgendiffGalois}
The parameterized generic Galois group $Gal^\partial(\cM_{\C(x)},\eta_{\C(x)})$ is the smallest
differential subgroup of $\GL_\nu(M_{\C(x)})$ that
contains a nonempty cofinite set of curvatures of the $q$-difference module $\cM_{K(x)}$.
\end{thm}

\section{The example of the Jacobi Theta function}

Consider the Jacobi Theta function
$$
\Theta(x)=\sum_{n\in\Z}q^{-n(n-1)/2}x^n,
$$
which is solution of the $q$-difference equation
$$
\Theta(qx)=qx\Theta(x).
$$
Iterating the equation, one proves that $\Theta$ satisfies
$y(q^nx)=q^{n(n+1)/2}x^ny(x)$, for any $n\geq 0$,
therefore we immediately deduce that the generic Galois group of the rank one
$q$-difference module $\cM_{\Theta}=(K(x).\Theta,\Sgq)$, with
$$
\begin{array}{rccc}
\Sgq&:K(x).\Theta&\longrightarrow & K(x).\Theta\\~\\
&f(x)\Theta&\longmapsto &f(qx)qx\Theta
\end{array},
$$
is the whole multiplicative group ${\mathbb G}_{m,K(x)}$.
As far as the parameterized generic Galois group is concerned we have:

\begin{prop}\label{prop:Thetagroups}
The parameterized generic Galois group
$Gal^{\partial}\l(\cM_\Theta, \eta_{K(x)}\r)$ is defined by
$\partial(\partial(y)/y)=0$.
\end{prop}

\begin{proof}
For almost any $v\in\cC$, the reduction modulo $\phi_v$
of $q^{\kappa_v(\kappa_v+1)/2}x^{\kappa_v}$ is the monomial
$x^{\kappa_v}$, which satisfies the equation $\partial\l(\frac{\partial x^{\kappa_v}}{x^{\kappa_v}}\r)=0$.
This means that parameterized generic Galois group $Gal^{\partial}\l(\cM_\Theta, \eta_{K(x)}\r)$ is a subgroup of
the differential group defined by $\partial\l(\frac{\partial y}{y}\r)=0$.
In other words, the logarithmic derivative
$$
\begin{array}{ccc}
\mathbb G_m & \longrightarrow & \mathbb G_a\\
y & \longmapsto & \frac{\partial y}{y}
\end{array}
$$
sends $Gal^{\partial}\l(\cM_\Theta, \eta_{K(x)}\r)$ into a subgroup of the additive group $\mathbb G_{a,K(x)}$ defined by the equation
$\partial z=0$. This is nothing else that $\mathbb G_{a,K}$.
Since $Gal^{\partial}\l(\cM_\Theta, \eta_{K(x)}\r)$ is not finite, it must be the whole group $\mathbb G_{a,K}$.
\end{proof}

Let us consider a norm $|~|$ on $K$ such that $|q|\neq 1$.
The differential dimension of the subgroup $\partial\l(\frac{\partial y}{y}\r)=0$ is zero.
We will show in \cite{diviziohardouinComp}
(see also \cite[\S9, Corollary 9.12]{diviziohardouin})
that this means that $\Theta$ is differentially algebraic over
the field of rational functions $\wtilde C_E(x)$ with coefficients in the differential closure
$\wtilde C_E$ of the elliptic function over $K^*/q^\Z$.
In fact, the function $\Theta$ satisfies
$$
\sgq\l(\frac{\partial\Theta}{\Theta}\r)=\frac{\partial\Theta}{\Theta}+1,
$$
which implies that $\partial\l(\frac{\partial\Theta}{\Theta}\r)$
is an elliptic function.
Since the Weierstrass function is differentially algebraic over $K(x)$,
the Jacobi Theta function is also differentially algebraic over $K(x)$.
\par
If $q$ is transcendental over $\Q$ we can also consider the derivation $\frac{d}{dq}$.
This case is studied in \cite{diviziohardouinPacific}.

\section{The Kolchin closure of the Dynamic and the  Malgrange-Granier groupoid}
\label{sec:malgrangegranieralg}

A. Granier has defined a $D$-groupoid for non linear $q$-difference equations, in analogy with Malgrange
$D$-groupoid for nonlinear differential equations.
In this section we prove that the
Malgrange-Granier groupoid in the special case of linear $q$-difference equations essentially ``coincides''
with the parameterized generic Galois group.
To prove this result we construct an algebraic $D$-groupoid called $\Gal{A}$
and we show on one hand that $\Gal{A}$ and the Malgrange-Granier groupoid $\Galan{A}$, which is an analytic object, have the same solutions, and, on the
other hand, that the solutions of the sub-$D$-groupoid of $\Gal{A}$,
that fixes the transversals, coincides with the solutions of the differential equations
defining the parameterized generic Galois group.
In the differential case, the problem of the algebraicity of the $D$-groupoid has been tackled
in more recent works by B. Malgrange himself.
\par
To prove the results below we have had to deal with some difficulties.
First of all, Malgrange proves that his $D$-groupoid, in the special case of a linear differential equation,
allows to recover the Picard-Vessiot Galois group (see \cite{MalgGGF}).
The foliation associated to the solutions of the nonlinear differential equation, which exists
due to the Cauchy theorem, plays a central role in his proof, and actually in the whole theory.
There is a true hindrance to prove a Cauchy theorem and define a foliation over $\C$ attached to a $q$-difference system.
First of all, the solutions of a $q$-difference equation must be defined over a $q$-invariant domain and
they usually have an essential singularity at $0$ and at $\infty$. This fact prevents the existence if a
local solution on a compact domain and therefore a transposition of the Cauchy theorem. 
In \cite{GranierFourier}, A. Granier defines the Galois $D$-groupoid of a $q$-difference system as the $D$-envelop of the dynamic of the system.
To overcome the lack of local solutions, we use Theorem \ref{thm:diffgenGalois} as a crucial ingredient of the proof.
We recover in this way the parameterized generic Galois group, rather than the Picard-Vessiot Galois group
(on this point see \S\ref{subsec:MalgrangeGranier} below).
However, some steps of the proof are similar to Malgrange theorem (\cf \cite{MalgGGF}).
We also recover Granier's result
for linear $q$-difference systems with constant coefficients
(\cf \cite[\S 2.1]{GranierFourier}).
\par
These results shall give some hints to compare the algebraic definitions of Morikawa
of the Galois group of a nonlinear
$q$-difference equation and the analytic definitions of A.Granier
(\cf \cite{morikawa}, \cite{morikawaumemura}, \cite{umemurapreprint}).

\medskip
Let $q \in \C^* $ be not a root of unity and let $A(x)\in \GL_\nu(\C(x))$.
We consider the linear $q$-difference system
\begin{equation}\label{eqn:qsyst}
Y(qx)=A(x)Y(x).
\end{equation}
We set:
$$
\begin{array}{l}
A_k(x):=A(q^{k-1}x)\dots A(qx)A(x)~\hbox{for all~}k\in\Z,\,k>0;\\
A_0(x)=Id_\nu\\
A_k(x):=A(q^{k}x)^{-1}A(q^{k+1}x)^{-1}\dots A(q^{-1}x)^{-1}~\hbox{for all~}k\in\Z,\,k<0,
\end{array}
$$
so that $Y(q^kx)=A_k(x)Y(x)$, for any $k\in\Z$.
Following the Appendix, we denote by $M$ the analytic complex variety $\P^1_\C \times \C^\nu$, by $\Galan{A(x)}$ the
Galois $D$-groupoid of the system (\ref{eqn:qsyst}) \ie the $D$-envelop of the dynamic
\beq\label{eq:dyn}
Dyn(A(x))=\l\{(x,X)\longmapsto (q^k x, A_k(x)X)\,:\,k\in\Z\r\}
\eeq
in the space of jets $J^*(M,M)$.
We keep the notation of \S\ref{sec:anDgroupoid}, which is preliminary to the content of this section.

\medskip
\noindent{\it Warning.}
Following Malgrange and the convention in \S\ref{subsec:defmalgrange},
we say that a $D$-groupoid $\cH$ is contained
in a $D$-groupoid $\cG$ if the groupoid of solutions of $\cH$ is contained in the groupoid of
solutions of $\cG$.
We will write
$sol(\cH)\subset sol(\cG)$ or equivalently $\cI_{\cG}\subset \cI_\cH$, where
$\cI_\cG$ and $\cI_\cH$ are the (sheaves of) ideals of definition of $\cG$ and $\cH$, respectively.

\paragraph{Notation.} In this section we introduce many tools that we use to get the proof of our main result
Corollary \ref{cor:Malgrangoide}. For the reader convenience we make a list of them here, with the reference for their definitions:

\begin{tabular}{llll}
$Dyn(A)$, & \eqref{eq:dyn};\\
$\Galan{A}$, &\S\ref{subsec:anagrouplin};
    &$\wtilde{\mathcal{G}al(A(x))}$, &Definition \ref{defn:tildegal};\\
$\Gal{A}$, &Definition \ref{defn:kol}; &$\Galt{A}$, &Definition \ref{defn:galt};\\
$\Kol{A}$, &Definition \ref{defn:kol}; &$\Kolt{A}$, &Defintion \ref{defn:kolt};\\
$\cL in$,  &Proposition \ref{prop:DgroupoideLin}; &$\cT rv$, &Defintion \ref{defn:trv}.
\end{tabular}


\subsection{The groupoid $\Gal{A}$}

Let $\C(x)\l\{ T,\frac{1}{\det T}\r\}_{\partial}$, with
$T=(T_{i,j}:i,j=0,1,\dots,\nu)$, be the algebra of differential rational functions
over $\GL_{\nu+1}(\C(x))$.
We consider the following morphism of $\partial$-differential $\C[x]$-algebras
$$
\begin{array}{rccc}
\tau:&\C[x]\l\{ T,\frac{1}{\det T}\r\}_{\partial} & \longrightarrow & H^0(M \times_\C M, \cO_{ J^*(M,M)}) \\~\\
&\begin{pmatrix}
  T_{0,0} & T_{0,1}& \hdots & T_{0,\nu} \\
  T_{1,0} & \\
  \vdots & & (T_{i,j})_{i,j} & \\
  T_{\nu,0} & \end{pmatrix}
& \longmapsto &
\begin{pmatrix}
  \frac{\partial \ol{x}}{\partial x} &\frac{\partial \ol{x}}{\partial X_1} & \hdots & \frac{\partial \ol{x}}{\partial X_\nu} \\
  \frac{\partial \ol{X_1}}{\partial x}& \\
  \vdots & & \l(\frac{\partial \ol{X_i}}{\partial X_j}\r)_{i,j} & \\
  \frac{\partial \ol{X_\nu}}{\partial x} &
  \end{pmatrix}
\end{array}
$$
from $\C[x]\l\{ T,\frac{1}{\det T}\r\}_{\partial}$ to the global sections $H^0(M \times_\C M, \cO_{ J^*(M,M)})$ of $\cO_{ J^*(M,M)}$,
that can be thought as the algebra of global partial differential equations over $M\times M$.
The image by $\tau$ of the differential ideal
$$
\cI=\l(T_{0,1},\dots,T_{0,\nu},T_{1,0},\dots,T_{\nu,0},\partial(T_{0,0})\r),
$$
that defines the linear differential group
$$
\l\{diag(\alpha, \beta(x)):=\begin{pmatrix} \alpha & 0 \\0 & \beta(x)\end{pmatrix}\,:\,\mbox{where} \;
\alpha \in \C^* \; \mbox{and} \; \beta(x) \in \GL_{\nu}(\C(x)) \r\},
$$
is contained in the ideal $\cI_{\cL in}$ defining the $D$-groupoid $\cL in $
(\cf Proposition \ref{prop:DgroupoideLin}).

\begin{defn}\label{defn:kol}
We call $\Kol{A}$ the smallest differential subvariety of
$\GL_{\nu+1}(\C(x))$, defined over $\C(x)$, which contains
$$
\l\{
diag(q^k, A_k(x)):=\begin{pmatrix} q^k & 0 \\0 & A_k(x)\end{pmatrix}\,:\,k\in\Z\r\},
$$
and has the following property:
if we call $I_{\Kol{A}}$ the differential ideal defining $\Kol{A}$ and $I_{\Kol{A}}^\p=I_{\Kol{A}}\cap\C[x]\l\{ T,\frac{1}{\det T}\r\}_{\partial}$, then the
(sheaf of) differential ideal $\langle\cI_{\cL in}, \tau(I_{\Kol{A}}^\p)\rangle$
generates a $D$-groupoid, that we will call $\Gal{A}$, in the space of jets $J^*(M,M)$.
\end{defn}

\begin{rmk}\label{rmk:malgrange}
The definition above requires some explanations:
\begin{itemize}
\item
The phrase ``smallest differential subvariety of $\GL_{\nu+1}(\C(x))$'' must be understood in the following way.
The ideal of definition of $\Kol{A}$ is the largest differential ideal of $\C(x)\l\{ T,\frac{1}{\det T}\r\}_{\partial}$
which admits the
matrices $diag(q^k, A_k(x))$ as solutions for any $k\in\Z$ and verifies the second requirement of the definition.
Then $I_{\Kol{A}}$ is radical and the Ritt-Raudenbush theorem (\cf Theorem \ref{thm:rittraudenbush} above)
implies that $I_{\Kol{A}}$ is finitely generated.
Of course, the $\C(x)$-rational points of $\Kol{A}$ may give very poor information on its structure, so we would rather speak of solutions in a differential
closure of $\C(x)$.
\item
The structure of $D$-groupoid has the following consequence on the points of $\Kol{A}$:
if $diag(\a, \be(x))$ and $diag(\ga, \de(x))$ are two matrices with entries in a differential extension
of $\C(x)$ that belong to $\Kol{A}$
then the matrix $diag(\a\ga,\be(\ga x) \de(x))$ belongs to $\Kol{A}$.
In other words, the set of local diffeomorphisms $(x,X)\mapsto(\a x,\be(x)X)$ of $M\times M$ such that
$diag(\a, \be(x))$ belongs to $\Kol{A}$ forms a set theoretic groupoid.
We could have supposed only that $\Kol{A}$ is a differential variety and the solutions of $\Kol{A}$
form a groupoid in the sense above, but this wouldn't have been enough. In fact, it is not known if a sheaf of differential ideals of $J^*(M,M)$ whose solutions
forms a groupoid is actually a $D$-groupoid (\cf Definition \ref{defn:Dgroupoid}, and in particular conditions (ii') and (iii')).
B. Malgrange told us that he can only prove this statement for a Lie algebra.
\end{itemize}
\end{rmk}

The differential variety $\Kol{A}$ is going to be a bridge between the parameterized generic Galois group and the Galois $D$-groupoid $\Galan{A}$
defined in the appendix, \emph{via} the following theorem.

\begin{defn}\label{defn:kolt}
Let $\cM_{\C(x)}^{(A)}:=(\C(x)^\nu, \Sgq: X \mapsto A^{-1}\sgq(X))$ be the $q$-difference module over $\C(x)$
associated to the system $Y(qx)=A(x)Y(x)$, where $\sgq(X)$ is defined componentwise. We call
$\Kolt{A}$ the differential group over $\C(x)$
defined by the differential ideal
$\langle I_{\Kol{A}},T_{0,0}-1\rangle$ in $\C(x)\l\{ T,\frac{1}{\det T}\r\}_{\partial}$.
\end{defn}

Notice that, as for the Zariski closure, the Kolchin closure does not commute with the intersection, therefore
$\Kolt{A}$ is not the Kolchin closure of $\{A_k(x)\}_k$.
Then we have:

\begin{thm}\label{thm:clotkol}
$Gal^{\partial}(\cM_{\C(x)}^{(A)}, \eta_{\C(x)})\cong\Kolt{A}$.
\end{thm}

\begin{rmk}
One can define in exactly the same way an algebraic subvariety $\Zar A$
of $\GL_{\nu+1}(\C(x))$ containing the dynamic of the system and such that
$$
\l\{(x,X)\mapsto(\a x,\be(x)X):diag(\a, \be(x))\in\Zar A\r\}
$$
is a subgroupoid
of the groupoid of diffeomorphisms of $M\times M$.
Then one proves in the same way that $\Zart{A}$ coincide with the generic Galois group,
introduced in \cite{diviziohardouinqGroth}.
\end{rmk}

\begin{proof}
Let $\cN=constr^{\partial}(\cM)$ be a construction of differential algebra of $\cM$.
We can consider:
\begin{itemize}
\item
The basis denoted by $constr^{\partial}(\ul{e})$ of $\cN$ built from
the canonical basis $\ul{e}$ of $\C(x)^\nu$,
applying the same constructions of differential algebra.

\item
For any $\beta \in \GL_\nu(\C(x))$, the matrix
$constr^{\partial}(\beta)$ acting on $\cN$ with respect to the basis $constr^{\partial}(\ul{e})$,
obtained from $\beta$ by functoriality.
Its coefficients lies in $\C(x)[\beta, \partial(\beta),...]$

\item
Any $\psi=(\alpha, \beta) \in \C^* \times \GL_\nu(\C(x))$ acts semilinearly
on $\cN$ in the following way:
$\psi\ul e=(constr^{\partial}(\beta))^{-1}\ul e$ and $\phi(f(x)n)=f(\a x)n$, for any $f(x)\in \C(x)$ and $n\in\cN$.
In particular,
$(q^k, A_k(x))\in \C^* \times \GL_\nu(\C(x)) $ acts as $\Sgq^k$ on $\cN$.
\end{itemize}
A sub-$q$-difference module $\cE$ of $\cN$ correspond to an invertible matrix
$F \in \GL_\nu(\C(x))$ such that
\begin{equation}\label{eqn:sev}
F(q^kx)^{-1} constr^{\partial}(A_k) F(x)
= \begin{pmatrix}* & *\\0 & * \end{pmatrix},
\,\hbox{for any $k\in\Z$}.
\end{equation}
Now, $(1, \beta)\in \C^*\times \GL_\nu(\C(x))$ stabilizes $\cE$ if and only if
\begin{equation}\label{eqn:stab}
F(x)^{-1} constr^{\partial}(\beta) F(x) =
\begin{pmatrix}* & *\\0 & * \end{pmatrix}.
\end{equation}
Equation (\ref{eqn:sev}) corresponds to a differential polynomial $L(T_{0,0}, (T_{i,j})_{i,j \geq 1})$ belonging to $\C(x)\l\{ T,\frac{1}{\det T}\r\}_{\partial}$
and having the property that $L(q^k, (A_k))=0$ for all $k \in \Z$. On the other hand
\eqref{eqn:stab} corresponds to $L(1, (T_{i,j})_{i,j \geq 1}))$.
It means that the solutions of the differential ideal $\langle I_{\Kol{A}} ,T_{0,0}-1\rangle\subset \C(x)\l\{ T,\frac{1}{\det T}\r\}_{\partial}$
stabilize all the sub-$q$-difference modules of all the constructions of differential algebra, and hence that
$$
\Kolt{A} \subset Gal^{\partial}(\cM_{\C(x)}, \eta_{\C(x)}).
$$
Let us prove the inverse inclusion.
In the notation of Theorem \ref{thm:complexmodulesgendiffGalois}, there exists a finitely generated extension
$K$ of $\Q$ and a $\sgq$-stable subalgebra $\cA$ of $K(x)$ of the forms considered in \S\ref{subsec:car0} such that:
\begin{enumerate}
\item
$A(x)\in \GL_\nu(\cA)$, so that it defines a $q$-difference module $\cM_{K(x)}$ over $K(x)$;
\item
$Gal^\partial(\cM_{K(x)}^{(A)},\eta_{K(x)})\otimes_{K(x)}\C(x)\cong Gal^\partial(\cM_{\C(x)}^{(A)},\eta_{\C(x)})$;
\item
$\Kol{A}$ is defined over $\cA$, \ie there exists a differential ideal $I$ in the differential ring $\cA \{T, \frac{1}{\det(T)} \}_{\partial}$ such that
$I$ generates $I_{\Kol{A}}$ in $\C(x)\l\{ T,\frac{1}{\det T}\r\}_{\partial}$.
\end{enumerate}
For any element $\wtilde L$ of the defining ideal of $\Kolt{A}$ over $\cA$,
there exists
$$
L(T_{0,0};T_{i,j},i,j=1,\dots,\nu)\in I \subset \cA\l\{T, \frac{1}{\det(T)}\r\}_{\partial},
$$
such that $L\in\cI_{\Kol{A}}$ and $\wtilde L =L(1;T_{i,j},i,j=1,\dots,\nu)$.
If $q$ is an algebraic number, other than a root of unity, or if $q$ is transcendental,
then for almost all places $v\in\cC$
we have
$$
\wtilde L(A_{\kappa_v})\equiv L(1, A_{\kappa_v})\equiv L(q^{\kappa_v},A_{\kappa_v})\equiv 0 \; \mbox{modulo} \; \phi_v.
$$
This shows that $\Kolt{A}$ is a differential subgroup of
$\GL_\nu(\C(x))$ which contains a nonempty cofinite set of
$v$-curvatures, in the sense of Theorem
\ref{thm:complexmodulesgendiffGalois}.
Therefore, $\Kolt{A}$
contains the parameterized generic Galois group of
$\cM_{\C(x)}^{(A)}$.
\end{proof}

\begin{defn}\label{defn:galt}
We call $\Galt{A}$ the $D$-groupoid on $M \times_\C M$ intersection of $\Gal{A}$ and $\cT rv$.
\end{defn}

It follows from the definition that the $D$-groupoid $\Galt{A}$ is generated by its global equations \ie by $\cL in$ and the image
of the equations of $\Kolt{A}$ by the morphism $\tau$.
Therefore we deduce from Theorem \ref{thm:clotkol} the following statement:

\begin{cor}\label{cor:galalg}
As a $D$-groupoid, $\Galt{A}$ is generated by its global sections, namely the
$D$-groupoid $\cL in$ and the image of the equations of $Gal^\partial(\cM^{(A)}_{\C(x)},\eta_{\C(x)})$
\emph{via} the morphism $\tau$.
\end{cor}

\begin{rmk}\label{rmk:galalg}
The corollary above says not only that a germ of diffeomorphism $(x,X)\mapsto(x,\be(x)X)$ of $M$ is solution of
$\Galt{A}$ if and only if $\be(x)$ is solution of the differential equations defining the
parameterized generic Galois group of $\cM^{(A)}_{\C(x)}=(\C(x)^\nu,X\mapsto A(x)^{-1}\sgq(X))$, but also that the
two differential defining ideals ``coincide''.
\end{rmk}

The $D$-groupoid  $\Galt{A}$ is a
differential analog of the $D$-groupoid generated by an algebraic group
introduced in \cite[Proposition 5.3.2]{MalgGGF} by B. Malgrange.

\subsection{The Galois $D$-groupoid $\Galan{A}$ of a linear $q$-difference system}

Since $Dyn(A(x))$ is contained in the solutions of $\Gal{A}$, we have
$$
sol(\Galan{A(x)})\subset sol(\Gal{A})
$$
and
$$
sol(\wtilde{\Galan{A(x)}})\subset sol(\Galt{A}).
$$
as already mentioned, the solution are to be found in some differential closure of
$\C(x),\partial$.

\begin{thm}\label{thm:Malgrangoide}
The solutions of the $D$-groupoid $\wtilde{\Galan{A(x)}}$ (resp. $\Galan{A(x)}$)
coincide with the solutions of $\Galt{A}$ (resp. $\Gal{A}$).
\end{thm}

Combining the theorem above with Corollary \ref{cor:galalg}, we immediately obtain:

\begin{cor}\label{cor:Malgrangoide}
The solutions of the $D$-groupoid $\wtilde{\Galan{A(x)}}$ are germs of diffeomorphisms
of the form $(x,X)\longmapsto (x, \be(x)X)$, such that $\be(x)$ is a solution of the differential equations
defining $Gal^\partial(\cM^{(A)}_{\C(x)},\eta_{\C(x)})$, and \emph{vice versa}.
\end{cor}

\begin{rmk}
The corollary above says that the solutions of $\wtilde{\Galan{A}}$ in a neighborhood of a transversal $\{x_0\}\times\C^\nu$
(\cf Proposition \ref{prop:galt} below), rational over a differential extension $\cF$ of $\C(x)$,
correspond one-to-one with the solutions $\be(x)\in \GL_\nu(\cF)$ of the
differential equations defining the parameterized generic Galois group.
\par
It does not say that the two defining differential ideals can be compared.
We actually don't prove that $\Galan{A}$ is an ``algebraic $D$-groupoid'' and therefore that
$\Gal{A}$ and $\Galan{A}$ coincide as $D$-groupoids.
\end{rmk}

\begin{proof}[Proof of Theorem \ref{thm:Malgrangoide}]
Let $\cI $ be the differential ideal of $\Galan{A(x)}$ in $\cO_{J^*(M,M)}$ and let $\cI_r$ be the subideal of $\cI$ order $r$.
We consider the morphism of analytic varieties given by
$$
\begin{array}{rccc}
\iota :& \P^1_\C \times \P^1_\C &\longrightarrow & M \times_\C M \\~\\
   & \l(x, \ol{x}\r) &\longmapsto & \l(x, 0,\ol{x},0\r)
\end{array}\,
$$
and the inverse image $\cJ_r:=\iota^{-1}\cI_r$ (resp. $\cJ:=\iota^{-1}\cI$) of the sheaf $\cI_r$ (resp. $\cI$) over $\P^1_\C \times \P^1_\C $.
We consider similarly to \cite[Lemma 5.3.3]{MalgGGF}, the evaluation $ev(\iota^{-1}\cI)$
at $X=\ol X=\frac{\partial^i\ol X}{\partial x^i}=0$ of the
equations of $\iota^{-1}\cI$ and we denote by $ev(\cI)$ the direct
image by $\iota$ of the sheaf $ev(\iota^{-1}\cI)$.

The following lemma is crucial in the proof of the Theorem \ref{thm:Malgrangoide}:

\begin{lemma}
\label{lemma:ev}
A germs of local diffeomorphism $(x, X) \mapsto (\alpha x, \beta(x) X)$ of $M$ is solution of $\cI$ if and only if it is solution of $ev(\cI)$.
\end{lemma}

\begin{proof}
First of all, we notice that $\cI$ is contained in $\cL in$. Moreover the solutions of
$\cI$, that are diffeomorphisms mapping a neighborhood of $(x_0,X_0)\in M$ to a neighborhood of $(\ol x_0,\ol X_0)$,
can be naturally continued to diffeomorphisms of a neighborhood of $x_0\times\C^\nu$ to a neighborhood of
$\ol x_0\times\C^\nu$.
Therefore it follows from the particular structure of the solutions of $\cL in$, that they are also solutions of $ev\l(\cI\r)$
(\cf Proposition \ref{prop:DgroupoideLin}).
\par
Conversely,
let the germ of diffeomorphism $\l(x, X\r) \mapsto \l(\alpha x, \beta\l(x\r) X\r)$ be a solution of $ev\l(\cI\r)$ and
$E \in \cI_r$.
It follows from Proposition \ref{LemmeReductionModLin}
that there exists $E_{1} \in \cI$ of order $r$,
only depending on the variables
$x$,$X$,$\frac{\partial \bar{x}}{\partial x}$,
$\frac{\partial \bar{X}}{\partial X}$,$\frac{\partial^{2} \bar{X}}{\partial x\partial X},\ldots$,
$\frac{\partial^{r}\bar{X}}{\partial x^{r-1}\partial X}$,
such that $\l(x, X\r) \mapsto \l(\alpha x, \beta\l(x\r) X\r)$ is solution of $E$ if and only if it is solution of $E_1$.
So we will focus on equations on the form $E_1$ and, to simplify notation, we will write $E$ for $E_1$.
\par
By assumption $\l(x, X\r) \mapsto \l(\alpha x, \beta\l(x\r) X\r)$ is solution of
$$
E\l( x,0,\frac{\partial \bar{x}}{\partial x},\frac{\partial \bar{X}}{\partial X},
\frac{\partial^{2} \bar{X}}{\partial x\partial X},\ldots \frac{\partial^{r}\bar{X}}{\partial x^{r-1}\partial X}\r)
$$
and we have to show that $\l(x, X\r) \mapsto \l(\alpha x, \beta\l(x\r) X\r)$ is a solution of $E$.
We consider the Taylor expansion of $E$:
$$
E\l( x,X,\frac{\partial \bar{x}}{\partial x},\frac{\partial \bar{X}}{\partial X},
\frac{\partial^{2} \bar{X}}{\partial x\partial X},\ldots \frac{\partial^{r}\bar{X}}{\partial x^{r-1}\partial X}\r)
=
\sum_{\alpha} E_{\alpha}\l(x, X\r) \partial^{\alpha},
$$
where $\partial^{\alpha}$ is a monomial in the coordinates $\frac{\partial \bar{x}}{\partial
x},\frac{\partial \bar{X}}{\partial X},\frac{\partial^{2} \bar{X}}{\partial x\partial X},\ldots \frac{\partial^{r}
\bar{X}}{\partial x^{r-1}\partial X}$.
Developing the $E_\a\l(x,X\r)$ with respect to $X=(X_1,\dots,X_\nu)$ we obtain:
$$
E = \sum \l(\sum_\alpha
\l(\frac{\partial^{\ul k} E_\alpha}{\partial X^{\ul k}}\r)\l(x,0\r) \partial ^\alpha \r) X^{\ul k},
$$
with $\ul k\in(\Z_{\geq 0})^\nu$.
If we show that for any $\ul k$ the germ $\l(x, X\r) \mapsto \l(\alpha x, \beta\l(x\r) X\r)$ verifies the equation
$$
B_{\ul k} := \sum_\alpha \l(\frac{\partial^{\ul k} E_\alpha}{\partial X^{\ul k}}\r) \l(x,0\r) \partial ^\alpha
$$
we can conclude.
For $\ul k=(0,\dots,0)$, there is nothing to prove since $B_{\ul 0}=ev\l(E\r)$.
\par
Let $D_{X_i}$ be the derivation of $\cI$ corresponding to $\frac{\partial}{\partial X_i}$,
The differential equation
$$
D_{X_i}\l(E\r) = \sum_\alpha \l(\frac{\partial E_\alpha}{\partial{X_i}}\r) \l(x,X\r) \partial ^\alpha + \sum_\alpha E_\alpha\l(x,X\r) D_{X_i}\l(\partial^\alpha\r)
$$
is still in $\cI$, since $\cI$ is a differential ideal.
Therefore by assumption $\l(x, X\r) \mapsto \l(\alpha x, \beta\l(x\r) X\r)$ is a solution of
$$
ev\l(D_{X_i}E\r)=
\sum_\alpha \l(\frac{\partial E_\alpha}{\partial{X_i}}\r) \l(x,0\r) \partial ^\alpha + \sum_\alpha E_\alpha\l(x,0\r) D_{X_i}\l(\partial^\alpha\r)
.
$$
Since $D_{X_i}\l(\partial^\alpha\r)\in\cL in$
and $\l(x, X\r) \mapsto \l(\alpha x, \beta\l(x\r) X\r)$ is a solution of $\cL in$,
we conclude that $\l(x, X\r) \mapsto \l(\alpha x, \beta\l(x\r) X\r)$ is a solution of
$$
\sum_\alpha \l(\frac{\partial E_\alpha}{\partial X}\r) \l(x,0\r) \partial ^\alpha
$$
and therefore of $B_1$.
Iterating the argument, one deduce that
$\l(x, X\r) \mapsto \l(\alpha x, \beta\l(x\r) X\r)$
is solution of $B_{\ul k}$ for any $\ul k\in(\Z_{\geq 0})^\nu$, which ends the proof of the lemma.
\end{proof}
We go back to the proof of Theorem \ref{thm:Malgrangoide}.
Lemma \ref{lemma:ev} proves that the solutions of $\Galan{A\l(x\r)}$ coincide with those
of the $D$-groupoid $\Gamma$ generated by $\cL in$ and $ev\l(\cI\r)$, defined on the open neighborhoods
of any $x_0\times\C^\nu\in M$. By intersection with the equation $\cT rv$, the same holds
for the transversal groupoids $\wtilde{\Galan{A\l(x\r)}}$ and $\wtilde{\Gamma}$.

Since $ \P^1_\C \times \P^1_\C $ and $M \times_\C M$ are locally compact and
$\cI_r$ is a coherent sheaf over $M \times_\C M$, the sheaf $\cJ_r$ is an analytic coherent sheaf over $\P^1_\C \times \P^1_\C $ and so is its quotient $ev(\iota^{-1}(\cI_{r}))$. By
\cite[Theorem 3]{Gagaser}, there exists
 an algebraic coherent sheaf
$\J_r$ over the projective variety $\P^1_\C \times \P^1_\C$ such that the analyzation of $\J_r$ coincides with $ev(\iota^{-1}(\cI_{r}))$. This implies that $ev\l(\cI\r)$ is generated by
algebraic differential equations
which by definition have the dynamic for solutions.

We thus have that the $sol(\Ga)=sol(\Galan{A})\subset sol(\Gal{A})$.
Since both $\Ga$ and $\Gal{A}$ are algebraic, the minimality of the variety
$\Kol{A}$ implies that
$sol(\Gal{A})\subset sol(\Gamma)$.
We conclude that the solutions of $\Galan{A}$ coincide with those $\Gal{A}$.
The same hold for $\wtilde{\Galan{A}}$, $\wtilde{\Gamma}$ and $\Galt{A}$).
This concludes the proof.
 \end{proof}

\subsection{Comparison with known results in \cite{MalgGGF} and \cite{GranierFourier}.}
\label{subsec:MalgrangeGranier}

In \cite{MalgGGF}, B. Malgrange proves that the Galois-$D$-groupoid of a linear differential equation
allows to recover, in the special case of a linear differential equation, the Picard-Vessiot Galois group over $\C$. This is not in contradiction with the result above, since:
\begin{itemize}
\item
due to the fact that local solutions of a linear differential equation form a $\C$-vector space
(rather than a vector space on the field of elliptic functions!), \cite[Proposition 4.1]{Katzbull} shows that
the generic Galois group and the Picard-Vessiot Galois group in the differential setting become isomorphic above a certain extension of the local ring.
For more details on the relation between the generic Galois group and the usual Galois group see
\cite[Corollary 3.3]{pillay}.

\item
it is not difficult to prove that, in the differential setting, the Picard-Vessiot
Galois group and parameterized Galois group with respect to $\frac{d}{dx}$ coincide.
\end{itemize}
Therefore B. Malgrange actually finds a parameterized generic Galois group, which is hidden in his construction.
The steps of the proof above are the same as in his proof, apart that, to compensate
the lack of good local solutions, we are obliged to use Theorem \ref{thm:diffgenGalois}.
Anyway, the application of Theorem \ref{thm:diffgenGalois} appears to be very natural, if one considers
how close the definition of the dynamic of a linear $q$-difference system and the definition of the curvatures are.

\medskip
In \cite{GranierFourier}, A. Granier shows that in the case of a $q$-difference system with constant coefficients
the groupoid that fixes the transversals in $\Galan{A}$ is the Picard-Vessiot Galois group, \ie an algebraic group defined over $\C$.
Once again, this is not in contradiction with our results.
In fact, under this assumption, it is not difficult to show that the parameterized generic Galois group is defined
over $\C$. Moreover the parameterized generic Galois groups and the generic Galois group coincide, in fact if $\cM$
is a $q$-difference module over $\C(x)$ associated with a constant $q$-difference system, it is easy to prove that
the prolongation functor $F_\partial$ acts trivially on $\cM$, namely $F_\partial(\cM)\cong \cM\oplus\cM$.
Finally, to conclude that the generic Galois group coincide with the usual one, it is enough to notice
that they are associated with the same fiber functor, or equivalently that they stabilize exactly the same objects.

\medskip
Because of these results, G. Casale and J. Roques have conjectured that ``for linear (q-)difference systems, the action of Malgrange
groupoid on the fibers gives the classical Galois groups'' (\cf \cite{casaleroques}).
In \emph{loc. cit.}, they give two proofs of their main integrability result:
one of the them relies on the conjecture.
Here we have proved that the Galois-$D$-groupoid allows to recover exactly the parameterized generic Galois group.
By taking the Zariski closure one can also recovers the algebraic generic Galois group.
The comparison theorems in \cite{diviziohardouinComp} (see also \cite[Part IV]{diviziohardouin})
imply that we can also recover the Picard-Vessiot Galois group
(\cf \cite{vdPutSingerDifference}, \cite{SauloyENS}),
performing a Zariski closure and a convenient field extension, and the parameterized Galois group (\cf \cite{HardouinSinger}),
performing a field extension.

\appendix


\section{The Galois $D$-groupoid of a $q$-difference system, by Anne Granier}
\label{sec:anDgroupoid}

We recall here the definition of the Galois $D$-groupoid of a $q$-difference system, and how to recover groups from it
in the case of a linear $q$-difference system. This appendix thus consists in a summary of Chapter 3 of
\cite{GranierThese}.

\subsection{Definitions}
\label{subsec:defmalgrange}

We need to recall first Malgrange's definition of $D$-groupoids, following \cite{MalgGGF} but specializing it to the
base space $\P^1_\mathbb{C} \times \mathbb{C}^\nu$ as in \cite{GranierThese} and \cite{GranierFourier}, and to explain
how it allows to define a Galois $D$-groupoid for $q$-difference systems.\\

Fix $\nu\in \mathbb{N}^*$, and denote by $M$ the analytic complex variety $\P^1_\mathbb{C} \times \mathbb{C}^\nu$. We call
\textit{local diffeomorphism of $M$} any biholomorphism between two open sets of $M$, and we denote by $Aut(M)$ the
set of germs of local diffeomorphisms of $M$. Essentially, a $D$-groupoid is a subgroupoid of
$Aut(M)$ defined by a system of partial differential equations.\\

Let us precise what is the object which represents the system of partial differential equations in this rough
definition.

A germ of a local diffeomorphism of $M$ is determined by the coordinates denoted by $(x,X)=(x,X_1,\ldots , X_\nu)$ of
its source point, the coordinates denoted by $(\bar{x},\bar{X})=(\bar{x},\bar{X}_1,\ldots , \bar{X}_\nu)$ of its target
point, and the coordinates denoted by $\frac{\partial \bar{x}}{\partial x},\frac{\partial \bar{x}}{\partial X_1},
\ldots ,\frac{\partial \bar{X}_1}{\partial x},\ldots ,\frac{\partial^2 \bar{x}}{\partial x^2}, \ldots$ which represent
its partial derivatives evaluated at the source point. We also denote by $\delta$ the polynomial in the coordinates above,
which represents the
Jacobian of a germ evaluated at the source point. We will allow us abbreviations for some sets of these coordinates,
as for example $\frac{\partial \bar{X}}{\partial X}$ to represent all the coordinates
$\frac{\partial\bar{X}_i}{\partial X_j}$ and $\partial\bar X$ for all the coordinates
$\frac{\partial\bar{X}_i}{\partial x_j}$,
$\frac{\partial\bar{X}_i}{\partial\bar x_j}$,
$\frac{\partial\bar{X}_i}{\partial X_j}$ and
$\frac{\partial\bar{X}_i}{\partial\bar X_j}$.

We denote by $r$ any positive integer. We call \textit{partial differential equation}, or only \textit{equation}, of
order $\leq r$ any fonction $E(x,X,\bar{x},\bar{X},\partial \bar{x},\partial \bar{X}, \ldots ,\partial^r
\bar{x},\partial^r \bar{X})$ which locally and holomorphically depends on the source and target coordinates, and
polynomially on $\delta^{-1}$ and on the partial derivative coordinates of order $\leq r$. These equations are
endowed with a sheaf structure on $M \times M$ which we denote by $\mathcal{O}_{J^*_r(M,M)}$. We then denote by
$\mathcal{O}_{J^{*}(M,M)}$ the sheaf of all the equations, that is the direct limit of the sheaves
$\mathcal{O}_{J_{r}^{*}(M,M)} $. It is endowed with natural derivations of the equations with respect to the source
coordinates. For example, one has: $D_{x}.\frac{\partial \bar{X}_{i}}{\partial X_j}=\frac{\partial ^{2}
\bar{X}_{i}}{\partial x \partial X_j}$.

We will consider the pseudo-coherent (in the sense of \cite{MalgGGF}) and differential ideal
\footnote{We will say everywhere differential ideal for sheaf of differential ideal.} $\mathcal{I}$ of
$\mathcal{O}_{J^{*}(M,M)}$ as the systems of partial differential equations in the definition of $D$-groupoid. A
\textit{solution} of such an ideal $\mathcal{I}$ is a germ of a local diffeomorphism $g : (M,a) \rightarrow (M,g(a))$ such that, for any equation $E$ of the fiber $\mathcal{I}_{(a,g(a))}$, the function defined by $(x,X) \mapsto
E((x,X),g(x,X),\partial g(x,X),\ldots )$ is null in a neighbourhood of $a$ in $M$. The set of solutions of
$\mathcal{I}$ is denoted by $sol(\mathcal{I})$.\\

The set $Aut(M)$ is endowed with a groupoid structure for the composition $c$ and the inversion $i$ of the germs of
local diffeomorphisms of $M$. We thus have to characterize, with the comorphisms $c^*$ and $i^*$ defined on
$\mathcal{O}_{J^{*}(M,M)}$, the systems of partial differential equations $\mathcal{I} \subset
\mathcal{O}_{J^{*}(M,M)}$ whose set of solutions $sol(\mathcal{I})$ is a subgroupoid of $Aut(M)$.

We call \textit{groupoid of order $r$} on $M$ the subvariety of the space of invertible jets of order $r$ defined by
a coherent ideal $\mathcal{I}_{r} \subset
\mathcal{O}_{J_{r}^{*}(M,M)}$ such that \textit{(i)}: all the germs of the identity map of $M$ are solutions of
$\mathcal{I}_{r}$, such that \textit{(ii)}: $c^{*}(\mathcal{I}_{r}) \subset \mathcal{I}_{r} \otimes
\mathcal{O}_{J_{r}^{*}(M,M)} + \mathcal{O}_{J_{r}^{*}(M,M)} \otimes \mathcal{I}_{r}$, and such that \textit{(iii)}:
$\iota ^{*}(\mathcal{I}_{r}) \subset \mathcal{I}_{r}$. The solutions of such an ideal $\mathcal{I}_{r}$ form a
subgroupoid of $Aut(M)$.

\begin{defn}\label{defn:Dgroupoid}
According to \cite{MalgGGF}, a \textit{$D$-groupoid} $\cG$ on $M$ is a
subvariety of the space $(M^2,\cO_{J^*(M,M)})$ of invertible jets defined by a reduced, pseudo-coherent and differential ideal
$\mathcal{I}_\cG \subset \mathcal{O}_{J^{*}(M,M)}$ such that
\begin{itemize}

\item[\textit{(i')}]
all the germs of the identity map of $M$ are solutions of $\mathcal{I}_\cG$,

\item[\textit{(ii')}]
for any relatively compact open set $U$ of $M$, there exists a closed complex analytic subvariety $Z$ of $U$ of codimension $\geq 1$, and a positive integer $r_{0} \in \mathbb{N}$ such that, for all $r \geq r_{0}$
and denoting by $\mathcal{I}_{\cG,r}= \mathcal{I}_\cG \cap \mathcal{O}_{J_{r}^{*}(M,M)}$,
one has, above $(U \setminus Z)^{2}$: $c^{*}(\mathcal{I}_{\cG,r}) \subset \mathcal{I}_{\cG,r} \otimes
\mathcal{O}_{J_{r}^{*}(M,M)} + \mathcal{O}_{J_{r}^{*}(M,M)} \otimes \mathcal{I}_{\cG,r}$,

\item[\textit{(iii')}]
$\iota ^{*}(\mathcal{I}_\cG) \subset \mathcal{I}_\cG$.
\end{itemize}
\end{defn}

The ideal $\cI_\cG$ totally determines the $D$-groupoid $\cG$, so we will rather focus on the ideal $\cI_\cG$ than its solution $sol(\cI_\cG)$ in $Aut(M)$.
Thanks to the analytic continuation theorem, $sol(\mathcal{I}_\cG)$ is a subgroupoid of $Aut(M)$.\\

The flexibility introduced by Malgrange in his definition of $D$-groupoid allows him to obtain two main results.
Theorem 4.4.1 of \cite{MalgGGF} states that the reduced differential ideal of $\mathcal{O}_{J^{*}(M,M)}$ generated by
a coherent ideal $\mathcal{I}_{r} \subset \mathcal{O}_{J_{r}^{*}(M,M)}$ which satisfies the previous conditions
\textit{(i)},\textit{(ii)}, and \textit{(iii)} defines a $D$-groupoid on $M$. Theorem 4.5.1 of \cite{MalgGGF} states
that for any family of $D$-groupoids on $M$ defined by a family of ideals $\{ \mathcal{G}^{i} \} _{i \in I}$, the ideal $\sqrt{\sum \mathcal{G}^{i}}$ defines a $D$-groupoid on $M$ called \textit{intersection}. The terminology is legitimated by the equality: $sol(\sqrt{\sum \mathcal{G}^{i}}) = \cap_{i \in I} sol(\mathcal{G}^{i})$. This last result allows to define the notion of $D$-envelope of any subgroupoid of $Aut(M)$.\\

Fix $q \in \mathbb{C}^*$, and let $Y(qx)=F(x,Y(x))$ be a (non linear) $q$-difference system, with $F(x,X) \in \mathbb{C}(x,X)^{\nu}$. Consider the set subgroupoid of $Aut(M)$ generated by the germs of the application $(x,X) \mapsto (qx,F(x,X))$ at any point of $M$ where it is well defined and invertible, and denote it by $Dyn(F)$. The Galois $D$-groupoid, also called the Malgrange-Granier groupoid in \S\ref{sec:malgrangegranieralg}, of the $q$-difference system $Y(qx)=F(x,Y(x))$ is the $D$-enveloppe of $Dyn(F)$, that is the \textit{intersection} of the $D$-groupoids on $M$ whose
set of solutions contains $Dyn(F)$.

\subsection{A bound for the Galois $D$-groupoid of a linear $q$-difference system}
\label{subsec:anagrouplin}

For all the following, consider a rational linear $q$-difference system $Y(qx)=A(x)Y(x)$, with $A(x) \in GL_{\nu}(\mathbb{C}(x))$. We denote by $\mathcal{G}al(A(x))$ the Galois $D$-groupoid of this system as defined at the end
of the previous section \ref{subsec:defmalgrange}, we denote by $\mathcal{I}_{\mathcal{G}al(A(x))}$ its defining ideal of equations, and by $sol(\mathcal{G}al(A(x)))$ its groupoid of solutions.\\

The elements of the dynamic $Dyn(A(x))$ of $Y(qx)=A(x)Y(x)$ are the germs of the local diffeomorphisms of $M$ of the form $(x,X) \mapsto (q^kx,A_k(x)X)$, with: $$A_{k}(x)= \left\lbrace \begin{array}{ll}  Id_{n} & \text{if } k=0,\\ \prod_{i=0}^{k-1} A(q^{i}x) & \text{if } k \in \mathbb{N}^{*},\\  \prod_{i=k}^{-1}A(q^{i}x)^{-1} & \text{if } k \in -\mathbb{N}^{*}.
\end{array} \right.$$
The first component of these diffeomorphisms is independent on the variables $X$ and depends linearly on the variable $x$, and the second component depends linearly on the variables $X$. These properties can be expressed in terms of partial differential equations. This gives an \textit{upper bound} for the Galois $D$-groupoid $\mathcal{G}al(A(x))$ which is defined in the following proposition.

\begin{prop}
\label{prop:DgroupoideLin}
The coherent ideal:
$$\left\langle \frac{\partial \bar{x}}{\partial X} , \frac{\partial \bar{x}}{\partial x}x-\bar{x} ,
\partial ^{2}\bar{x}, \frac{\partial \bar{X}}{\partial X}X-\bar{X} , \frac{\partial ^{2} \bar{X}}{\partial X^{2}}
\right\rangle \subset \mathcal{O}_{J^*_2(M,M)}$$
satisfies the conditions \textit{(i)},\textit{(ii)}, and \textit{(iii)} of \ref{subsec:defmalgrange}. Hence, thanks to Theorem 4.4.1 of \cite{MalgGGF}, the reduced differential ideal $\mathcal{I}_{\mathcal{L}in}$ it generates defines a $D$-groupoid $\mathcal{L}in$. Its solutions $sol(\mathcal{L}in)$ are the germs of the local diffeomorphisms of $M$ of the form: $$(x,X) \mapsto (\alpha x,\beta (x)X),$$ with $\alpha \in \mathbb{C}^*$ and locally, $\beta (x) \in GL_\nu(\mathbb{C})$ for all $x$.\\ They contain $Dyn(A(x))$, and therefore, given the definition of $\mathcal{G}al(A(x))$, one has the inclusion $$\mathcal{G}al(A(x)) \subset \mathcal{L}in,$$ which means that:
$$\mathcal{I}_{\mathcal{L}in} \subset \mathcal{I}_{\mathcal{G}al(A(x))} \, \, \text{ and } \, \,
sol(\mathcal{G}al(A(x))) \subset sol(\mathcal{L}in).$$
\end{prop}

\begin{proof}
\textit{cf} proof of Proposition 3.2.1 of \cite{GranierThese} for more details.
\end{proof}

\begin{rmk}
Given their shape, the solutions of $\mathcal{L}in$ are naturally defined in neighborhoods of transversals $\left\lbrace x_a \right\rbrace \times \mathbb{C}^{\nu}$ of $M$. Actually, consider a particular element of $sol(\mathcal{L}in)$, that is precisely a germ at a point $(x_a,X_a) \in M$ of a local diffeomorphism $g$ of $M$ of the form $(x,X) \mapsto (\alpha x,\beta (x)X)$. Consider then a neighborhood $\Delta$ of $x_a$ in $P^1\mathbb{C}$ where the matrix $\beta (x)$ is well defined and invertible, consider the "cylinders" $T_s=\Delta \times \mathbb{C}^{\nu}$ and $T_t=\alpha \Delta \times \mathbb{C}^{\nu}$ of $M$, and the diffeomorphism $\tilde{g} : T_s \rightarrow T_t$ well defined by $(x,X) \rightarrow (\alpha x,\beta (x)X)$. Therefore, according to the previous Proposition \ref{prop:DgroupoideLin}, all the germs of $\tilde{g}$ at the points of $T_s$ are in $sol(\mathcal{L}in)$ too.
\end{rmk}

The defining ideal $\mathcal{I}_{\mathcal{L}in}$ of the bound $\mathcal{L}in$ is generated by very simple equations. This allows to reduce modulo $\mathcal{I}_{\mathcal{L}in}$ the equations of $\mathcal{I}_{\mathcal{G}al(A(x))}$ and obtain some
simpler representative equations, in the sense that they only depend on some variables.

\begin{prop}
\label{LemmeReductionModLin}
Let $r \geq 2$. For any equation $E \in \mathcal{I}_{\mathcal{G}al(A(x))}$ of order $r$, there
exists an invertible element $u \in \mathcal{O}_{J^{*}_r(M,M)}$, an equation $L \in
\mathcal{I}_{\mathcal{L}in}$ of order $r$, and an equation $E_{1} \in \mathcal{I}_{\mathcal{G}al(A(x))}$ of order $r$
only depending on the variables written below, such that:
$$
uE=L+E_{1}
\l(x,X,\frac{\partial \bar{x}}{\partial
x},\frac{\partial \bar{X}}{\partial X},\frac{\partial^{2} \bar{X}}{\partial x\partial X},\ldots \frac{\partial^{r}
\bar{X}}{\partial x^{r-1}\partial X}\r).
$$
\end{prop}

\begin{proof}
The invertible element $u$ is a good power of $\delta$. The proof consists then in performing the divisions of
the equation $uE$, and then its succesive remainders, by the generators of $\mathcal{I}_{\mathcal{L}in}$.
More details are given in the proof of Proposition 3.2.3 of \cite{GranierThese}.
\end{proof}

\subsection{Groups from the Galois $D$-groupoid of a linear $q$-difference system}

We are going to prove that the solutions of the Galois $D$-groupoid $\mathcal{G}al(A(x))$ are, like the solutions of the bound $\mathcal{L}in$, naturally defined in neighbourhoods of transversals of $M$. This property, together with the
groupoid structure of $sol(\mathcal{G}al(A(x)))$, allows to exhibit groups from the solutions of $\mathcal{G}al(A(x))$
which fix the transversals.\\

According to Proposition \ref{prop:DgroupoideLin}, an element of $sol(\mathcal{G}al(A(x)))$ is also an element of $sol(\mathcal{L}in)$. Therefore, it is a germ at a point $a=(x_a,X_a) \in M$ of a local diffeomorphism $g :(M,a)
\rightarrow (M,g(a))$ of the form $(x,X) \mapsto (\alpha x,\beta (x)X)$, such that, for any equation $E
\in \mathcal{I}_{\mathcal{G}al(A(x))}$, one has $E((x,X),g(x,X),\partial g(x,X),\ldots )=0$ in a neighbourhood
of $a$ in $M$.

Consider an open connected neighbourhood $\Delta$ of $x_a$ in $\P^1_\mathbb{C}$ on which the matrix $\beta$ is
well-defined and invertible, that is where $\beta$ can be prolongated in a matrix $\beta \in GL_\nu(
\mathcal{O}(\Delta))$. Consider the "cylinders" $T_s=\Delta \times \mathbb{C}^\nu$ and $T_t=\alpha \Delta \times \mathbb{C}^\nu$ of $M$, and the
diffeomorphism $\tilde{g} : T_s \rightarrow T_t$ defined by $(x,X) \rightarrow (\alpha x,\beta (x)X)$.

\begin{prop}
\label{DefSolTransversale}
The germs at all points of $T_s$ of the diffeomorphism $\tilde{g}$ are elements of $sol(\mathcal{G}al(A(x)))$.
\end{prop}

\begin{proof}
For all $r \in \mathbb{N}$, the ideal $(\mathcal{I}_{\mathcal{G}al(A(x))})_{r} = \mathcal{I}_{\mathcal{G}al(A(x))} \cap \mathcal{O}_{J^*_{r}(M,M)}$ is coherent. Thus, for any point $(y_0,\bar{y}_0) \in M^2$, there exists an open neighbourhood $\Omega$ of
$(y_0,\bar{y}_0)$ in $M^2$, and equations $E_1^{\Omega}, \ldots ,E_l^{\Omega}$ of $(\mathcal{I}_{\mathcal{G}al(A(x))})_{r}$ defined on
the open set $\Omega$ such that: $$\left( (\mathcal{I}_{\mathcal{G}al(A(x))})_{r} \right) _{|\Omega }=\left(
\mathcal{O}_{J^*_{r}(M,M)} \right) _{|\Omega } E_1^{\Omega} + \cdots + \left( \mathcal{O}_{J^*_{r}(M,M)} \right)
_{|\Omega } E_l^{\Omega}.$$
Let $a_1 \in T_s=\Delta \times \mathbb{C}^{\nu}$. Let $\gamma : [0,1] \rightarrow T_s$ be a path in $T_s$ such that
$\gamma (0)=a$ and $\gamma (1)=a_1$. Let $\left\lbrace \Omega_0, \ldots ,\Omega_N \right\rbrace$ be a finite covering
of the path $\gamma ([0,1]) \times \tilde{g}(\gamma ([0,1]))$ in $T_s \times T_t$ by connected open sets $\Omega
\subset (T_s \times T_t)$ like above, and such that the origin $(\gamma (0),g(\gamma (0)))=(a,g(a))$ belongs to
$\Omega_0$.\\
The germ of $g$ at the point $a$ is an element of $sol(\mathcal{G}al(A(x)))$. Therefore, one has
$E_k^{\Omega_0}((x,X),g(x,X),\partial g(x,X),\ldots )$ $\equiv 0$ in a neighbourhood of $a$ for all $1 \leq l
\leq k$. Moreover, by analytic continuation, one has also $E_k^{\Omega_0}(x,X,\tilde{g}(x,X),$ $\partial
\tilde{g}(x,X),\ldots ) \equiv 0$ on the source projection of $\Omega_0$ in $M$. It means that the germs of
$\tilde{g}$ at any point of the source projection of $\Omega_0$ are solutions of $(\mathcal{I}_{\mathcal{G}al(A(x))})_{r}$.\\
Then, step by step, one gets that the germs of $\tilde{g}$ at any point of the source projection of $\Omega_k$ are
solutions of $(\mathcal{I}_{\mathcal{G}al(A(x))})_{r}$ and, in particular, the germ of $\tilde{g}$ at the point $a_1$ is also a solution of $(\mathcal{I}_{\mathcal{G}al(A(x))})_{r}$.
\end{proof}

\noindent
This Proposition \ref{DefSolTransversale} means that any solution of the Galois $D$-groupoid
$\mathcal{G}al(A(x))$ is naturally defined in a neighbourhood of a transversal of $M$, above.

\begin{rmk}
In some sense, the "equations" counterpart of this proposition is Lemma \ref{lemma:ev}.
\end{rmk}

The solutions of $\mathcal{G}al(A(x))$ which fix the transversals of $M$ can be interpreted as solutions of a
sub-$D$-groupoid of $\mathcal{G}al(A(x))$, partly because this property can be interpreted in terms of partial
differential equations. Actually, a germ of a diffeomorphism of $M$ fix the transversals of $M$ if and only if it is a
solution of the equation $\bar{x}-x$.

The ideal of $\mathcal{O}_{J^*_0(M,M)}$ generated by the equation $\bar{x}-x$ satisfies the conditions
\textit{(i)},\textit{(ii)}, and \textit{(iii)} of \ref{subsec:defmalgrange}. Hence, thanks to Theorem 4.4.1 of \cite{MalgGGF}, the reduced differential ideal it generates defines a $D$-groupoid:

\begin{defn}\label{defn:trv}
We call $\mathcal{T}rv$ the $D$-groupoid generated by the equation $\bar{x}-x$.
\end{defn}

Its solutions, $sol(\mathcal{T}rv)$, are the germs of the local diffeomorphisms of $M$ of the form: $(x,X) \mapsto (x,\bar{X}(x,X))$.

\begin{defn}\label{defn:tildegal}
We call $\wtilde{\mathcal{G}al(A(x))}$ the \textit{intersection} $D$-groupoid  $\mathcal{G}al(A(x)) \cap \mathcal{T}rv$,
in the sense of Theorem 4.5.1 of \cite{MalgGGF}, whose defining ideal of equations $\mathcal{I}_{\wtilde{\mathcal{G}al(A(x))}}$ is generated by $\mathcal{I}_{\mathcal{G}al(A(x))}$ and $\mathcal{I}_{\mathcal{T}rv}$.
\end{defn}

The solutions of $sol(\wtilde{\mathcal{G}al(A(x))})$ coincide with $sol(\mathcal{G}al(A(x))) \cap sol(\mathcal{T}rv)$, that are exactly the solutions of $\mathcal{G}al(A(x))$ of the form $(x,X) \mapsto (x,\beta (x)X)$. They are also naturally defined in neighbourhoods of transversals of $M$.

\begin{prop}\label{prop:galt}
Let $x_0 \in \P^1_\mathbb{C}$. The set of solutions of $\wtilde{\mathcal{G}al(A(x))}$ defined in a neighbourhood of the
transversal $\left\lbrace x_0 \right\rbrace \times \mathbb{C}^\nu$ of $M$ can be identified with a
subgroup of $GL_\nu(\mathbb{C}\left\lbrace x-x_0 \right\rbrace )$.
\end{prop}

\begin{proof}
The solutions of the $D$-groupoid $\wtilde{\mathcal{G}al(A(x))}$ defined in a neighbourhood of the transversal
$\left\lbrace x_0 \right\rbrace \times \mathbb{C}^\nu$ can be considered, without loosing any information, only in a
neighbourhood of the stable point $(x_0,0) \in M$. At this point, the groupoid structure of
$sol(\wtilde{\mathcal{G}al(A(x)))}$ is in fact a group structure because the source and target points are always $(x_0,0)$. Thus, considering the matrices $\beta(x)$ in
the solutions $(x,X) \mapsto (x,\beta(x)X)$ of $\wtilde{\mathcal{G}al(A(x)))}$ defined in a neighbourhood of $\left\lbrace x_0 \right\rbrace \times
\mathbb{C}^\nu$, one gets a subgroup of $GL_\nu(\mathbb{C}\left\lbrace x-x_0 \right\rbrace )$. More details are given in
the proof of Proposition 3.3.2 of \cite{GranierThese}.
\end{proof}

In the particular case of a constant linear $q$-difference system, that is with $A(x)=A \in GL_\nu(\mathbb{C})$, the
solutions of the Galois $D$-groupoid $\mathcal{G}al(A)$ are in fact global diffeomorphisms of $M$, and the set of those that fix the transversals of $M$ can be identified with an algebraic subgroup of $GL_\nu(\mathbb{C})$. This can be shown using a better bound than $\mathcal{L}in$ for the Galois $D$-groupoid of a constant linear $q$-difference system (\textit{cf} Proposition 3.4.2 of \cite{GranierThese}), or computing the $D$-groupoid $\mathcal{G}al(A)$ directly (\textit{cf} Theorem 2.1 of \cite{GranierFourier} or Theorem 4.2.7 of \cite{GranierThese}). Moreover, the explicit computation allows to observe that this subgroup corresponds to the usual $q$-difference Galois group as described in
\cite{SauloyENS} of the constant linear $q$-difference system $X(qx)=AX(x)$
(\cf Theorem 4.4.2 of \cite{GranierThese} or Theorem 2.4 of \cite{GranierFourier}).


\newcommand{\noopsort}[1]{}

\end{document}